\documentclass[11pt]{article}

\usepackage{times}
\usepackage{float}
\usepackage{bbm}
\usepackage{amsfonts,amstext,amsmath,amssymb,amsthm}
\usepackage{mathrsfs}
\usepackage{graphicx}
\usepackage{epstopdf}
\usepackage[hang]{caption2}

\renewcommand{\captionlabelfont}{\bfseries\it}
\renewcommand{\captionlabelfont}{\bfseries\emph} 
\long\def\symbolfootnote[#1]#2{\begingroup
\def\thefootnote{\fnsymbol{footnote}}\footnote[#1]{#2}\endgroup}
\usepackage{titlesec} 

\titleformat{\section}{\large\bfseries\uppercase}{\thesection.}{1em}{}
\titlespacing*{\section}{0pt}{*3}{*2}
\titleformat{\subsection}{\normalfont\bfseries}{\thesubsection.}{.5em}{}
\titlespacing*{\subsection}{0pt}{*3}{*2}
\titleformat{\subsubsection}{\normalfont\bfseries}{\thesubsubsection.}{.5em}{}
\titlespacing*{\subsubsection}{0pt}{*3}{*2}

\usepackage{nameref}

\usepackage[%
    colorlinks=true,%
    bookmarksnumbered=true,%
    bookmarksopen=true,%
    citecolor=blue,%
    urlcolor=blue,%
    unicode=true,           
    breaklinks=true,                                                        
    pdftitle={Nearly Minimax Mixture Rules for One-sided Sequential Testing},
    pdfauthor={Georgios Fellouris and  Alexander G. Tartakovsky},
]{hyperref}

\usepackage[authoryear]{natbib}

\numberwithin{equation}{section}

\newcommand{\cJ}{\mathcal{J}}

\newcommand{\cL}{\mathcal{L}}
\newcommand{\cF}{\mathcal{F}}

\newcommand{\cB}{\mathcal{B}}

\newcommand{\cH}{\mathcal{H}}
\newcommand{\cP}{\mathcal{P}}
\newcommand{\cR}{\mathcal{R}}
\newcommand{\cca}{\mathcal{C}_{\alpha}}
\newcommand{\ccab}{\mathcal{C}_{\alpha, \beta}}

\newcommand{\Fc}{{ \mathscr{F}}}

\newcommand{\Exp}{{\sf E}}
\newcommand{\Expop}{{\mathbf{E}}}

\newcommand{\Pro}{{\sf P}}
\newcommand{\Prop}{{\mathbf{P}}}
\newcommand{\Hyp}{{\sf H}}

\newcommand{\set}[1]{\left\{#1\right\}}
\newcommand{\brc}[1]{\left(#1\right)}
\newcommand{\brcs}[1]{\left[#1\right]}

\theoremstyle{plain} 
\newtheorem{theorem}{Theorem}[section]
\newtheorem{lemma}{Lemma}[section]
\newtheorem{corollary}{Corollary}[section]

\theoremstyle{remark}
\newtheorem{remark}{Remark}[section]
\newtheorem*{remark*}{Remark}

\theoremstyle{definition}

\newtheorem*{definition*}{Definition}

\newtheorem{condition}{Condition}

\newcommand{\ind}[1]{\mathbbm{1}_{\{#1\}}}  
\usepackage{color}
\definecolor{deprecated}{rgb}{0.5,0.5,0.5}

\begin{document}

\title{\textbf{\Large Nearly Minimax One-Sided Mixture-Based Sequential Tests}}  

\date{}
\maketitle

\begin{center}
\null\vskip-2cm\author{
\textbf{\large Georgios Fellouris and  Alexander G. Tartakovsky }\\
Department of Mathematics, University of Southern California\\
Los Angeles, California, USA}
\end{center}

\symbolfootnote[0]{\normalsize \hspace{-0.7cm} Address correspondence to A.G. Tartakovsky, Department of Mathematics, University of Southern California, KAP-108, Los Angeles, CA 90089-2532, USA; E-mail: tartakov@math.usc.edu or G. Fellouris Department of Mathematics, University of Southern California, KAP-108, Los Angeles, CA 90089-2532, USA; E-mail: fellouri@usc.edu.}

{\small \noindent\textbf{Abstract:} We focus on one-sided, mixture-based stopping rules for the problem of sequential testing a simple null hypothesis against a composite alternative. For the latter, we consider two cases---either a discrete alternative  or a continuous alternative that can be embedded into an exponential family. For each case, we find a mixture-based stopping rule that is nearly minimax in the sense of minimizing the maximal Kullback--Leibler information. The proof of this result is based on finding an almost Bayes rule for an appropriate sequential decision problem and on high-order asymptotic approximations for the performance characteristics of arbitrary mixture-based stopping times. We also evaluate the asymptotic performance loss of certain intuitive mixture rules and  verify the  accuracy of our asymptotic approximations with simulation experiments.}

\\


\section{INTRODUCTION}

\subsection{Problem Formulation and Literature Review}

Let  $\{X_{n}\}_{n \in \mathbb{N}}$ be a sequence of independent and identically distributed (iid)  observations (generally vectors, $X_n \in \mathbb{R}^d$) whose common distribution under the probability measure $\Pro_{0}$ (the \textit{null} hypothesis $\Hyp_0: \Pro=\Pro_0$) is $F_{0}$.
There is no cost for sampling under $\Pro_{0}$. However sampling should be terminated as soon as possible if there is sufficient evidence against $\Pro_{0}$ and in favor of a class of probability measures $\cP$ (an \textit{alternative} hypothesis $\Hyp: \Pro \in \cP$). The problem is to find an $\{\Fc_{n}\}$-stopping time that takes large values under $\Pro_{0}$  and small values under every probability measure  in $\cP$, where $\Fc_{n}=\sigma(X_{1},\ldots, X_{n})$ is the sigma-algebra generated by the first $n$ observations $X_1,\dots,X_n$, $n \ge 1$. 

When $\cP$ consists  of a single probability measure, say $\cP=\{\Pro_{1}\}$, and  the $\Pro_{1}$-distribution of $X_{1}$, $F_{1}$, is absolutely continuous with respect to $F_{0}$, a definitive solution to this sequential hypothesis testing problem is the \textit{one-sided} Sequential Probability Ratio Test (SPRT)
$$
T^{1}_{A} = \inf \{n \geq 1: \Lambda^{1}_{n} \geq A \}, \quad \inf\{\varnothing\} = \infty,
$$
where $A>1$ is a fixed level (threshold) and  $\{\Lambda_{n}^{1}\}$ is the corresponding likelihood-ratio process, i.e.,
$$
\Lambda_{n}^{1} = \prod_{m=1}^{n} \frac{dF_{1}}{dF_{0}} (X_{m}), \quad n \in \mathbb{N}.
$$
The stopping time $T^{1}_{A}$ is often called an \textit{open-ended} test or a test  of \textit{power one}, because it does not terminate almost surely under $\Pro_{0}$ ($\Pro_{0}(T^{1}_{A}<\infty) \leq 1/A$), whereas  it terminates almost surely under $\cP$, i.e., $\Pro_{1}(T^{1}_{A}<\infty)=1$. Furthermore, it follows from  \citet[pp.\ 107--108]{ChowRobbinsSiegmund-book71} that  if the threshold $A=A_\alpha$ is selected so that $\Pro_0(T_A^1 < \infty)=\alpha$, then 
\begin{equation} \label{optimum}
\Exp_{1}[T^{1}_{A}] =\inf_{T \in \cca} \Exp_{1}[T],
\end{equation}
where $\Exp_1$ denotes expectation with respect to $\Pro_1$ and $\cca=\{T:\Pro_{0}(T<\infty)\leq \alpha\}$ is the class  of stopping times whose ``error probability'' is bounded by $\alpha$, $0 < \alpha < 1$.

When the alternative hypothesis is not simple, there  have been extensions of the one-sided SPRT, but none of them exhibits such an exact optimality property as (\ref{optimum}) under every probability measure associated with the alternative hypothesis $\cP$.  More specifically, suppose that $\cP=\{\Pro_{\theta}\}_{\theta  \in \Theta \backslash \{0\}}$  and that the $\Pro_{\theta}$-distribution of $X_{1}$  belongs to the exponential family
\begin{equation} \label{koop}
\frac{dF_{\theta}(x)}{dF_{0}(x)}= e^{\theta x- \psi_\theta}, \quad    \theta \in \Theta = \{\theta \in \mathbb{R}: \Exp_{0}[e^{\theta X_{1}}]<\infty\} ,
\end{equation}
where $\psi_{\theta}=\log \Exp_{0}[e^{\theta X_{1}}]$. Moreover, let $\Lambda_n^\theta$ be the likelihood ratio of $\Pro_{\theta}$ versus $\Pro_{0}$ based on the first $n$ observations, i.e.,
\begin{equation} \label{LR}
\Lambda_n^\theta =  \prod_{k=1}^n  \frac{d F_\theta(X_k)}{d F_0(X_k)} = \exp \Big\{\theta \, \sum_{k=1}^{n} X_{k}-  n \, \psi(\theta) \Bigr\}, \quad n \in \mathbb{N}
\end{equation} 
and let $I_{\theta}=  \Exp_\theta [ \log \Lambda_{1}^{\theta}]$ denote the Kullback--Leibler divergence of $F_{\theta}$ versus $F_{0}$, where here and in what follows $\Exp_\theta$ stands for expectation with respect to $\Pro_\theta$.

 A natural generalization of the one-sided SPRT is  the threshold stopping time $\inf \{n \geq 1: \Lambda^{\theta_{n}}_{n} \geq A \}$, where 
$\theta_{n}$ is an estimate of the unknown parameter $\theta$ at time $n$. \citet{Lorden-AS73} followed a generalized likelihood ratio approach, where
$\theta_{n}$ is taken to be the maximum likelihood estimator (MLE) of $\theta$ based on the  first $n$ observations (see also \citet{Lai-ss01} for two composite hypotheses and two-sided tests). \citet{RobbinsSiegmund-Berkeley70,RobbinsSiegmund-AS74} followed a non-anticipating estimation approach and considered  $\theta_{n}$ to be a one-step delayed estimator that depends only on the first $n-1$ observations. For the latter approach, we also refer to \citet{PollakYakir-SQA99}, \citet{Pavlov-TPA90}, \citet{DragNov}, and \citet{LordenPollak-AS05}. 

An alternative, mixture-based approach was used by \citet{DarlingRobbins-NAS68} (see also \citet{Robbins-AMS70}), 
where the stopping rule has the form
\begin{equation} \label{rule}
T_{A} = \inf \{n \geq 1: \Lambda_{n} \geq A \}
\end{equation}
with $\{\Lambda_{n}\}$ being a weighted (mixed) likelihood-ratio statistic given by 
\begin{equation} \label{cont}
\Lambda_{n} = \int_{\Theta} \Lambda_n^\theta  \; G(d\theta) \; ,  \quad n \in \mathbb{N}
\end{equation}
and $G$ being an arbitrary distribution function on $\Theta$. Assuming that $G$ has a positive and continuous density with respect to the Lebesgue measure, \citet{PollakSiegmund-AS75} obtained an  asymptotic approximation for $\Exp_{\theta}[T_{A}]$ as $A \rightarrow \infty$. 
Based on this approximation, \citet{Pollak-AS78} proved that if $\alpha=1/A$ and $\bar{\Theta} \subset \Theta$ is an arbitrary, closed, 
finite interval, bounded away from 0, then 
\begin{equation}  \label{polopt}
\inf_{T \in \cca} \;  \sup_{\theta \in \bar{\Theta}}  \;  I_{\theta} \, \Exp_{\theta}[T] \geq |\log \alpha| + \log \sqrt{|\log \alpha|}  +O(1) \quad \text{as}~ \alpha \to 0,
\end{equation}
where $O(1)$ is bounded as $\alpha \to 0$, and that this asymptotic lower bound is attained by \textit{any} mixture rule whose mixing distribution has a positive and continuous density  with support that includes $\bar{\Theta}$.  Note that $I_\theta \Exp_\theta[T] = \Exp_\theta[\log \Lambda_T^\theta]$ is the total 
Kullback--Leibler information in the trajectory $X_1^T=(X_1,\dots,X_T)$ in favor of the hypothesis $\Hyp_\theta: \Pro=\Pro_\theta$ versus $\Hyp_0: \Pro=\Pro_0$, so that the problem of minimizing of the maximal value of $I_\theta\Exp_\theta[T]$ can be interpreted as minimizing the Kullback--Leibler information in the least favorable situation.  

\citet{Lerche-AS86} considered the problem of sequential testing for the drift of a Brownian motion in a Bayesian setup.

\subsection{Main Contributions}

One of the goals of this work is to extend the above work on mixture rules. In the framework of exponential families, 
we show that a particular choice of the mixing density leads to a mixture rule $T_A$ that attains $\inf_{T \in \cca} \sup_{\theta \in \bar{\Theta}} \, (I_{\theta} \, \Exp_{\theta}[T])$, not only up to an $O(1)$ term  as in \citet{Pollak-AS78}, but up to an o(1) term  (see Theorem \ref{theo3}). 

However, the main emphasis is on the case that the alternative hypothesis $\cP$ is a finite set, $\cP=\{\Pro_{1}, \ldots, \Pro_{K}\}$. 
In this setup, the weighted likelihood ratio statistic becomes
\begin{equation} \label{disc}
\Lambda_{n}= \sum_{i=1}^{K} p_{i} \, \Lambda_{n}^{i},  \quad n \in \mathbb{N},
\end{equation}
where $\Lambda_{n}^{i}=\prod_{m=1}^{n} [dF_{i}(X_m)/dF_{0}(X_{m})]$,  $F_{i}$ is the $\Pro_{i}$-distribution of $X_{1}$, which is assumed to be absolutely continuous with respect to $F_{0}$, and $\{p_{i}\}$ is a probability mass function, i.e., $p_{i} \geq 0$ for every $i$ and $\sum_{i=1}^{K} p_{i}=1$.  This is a more general framework than that of an exponential family, in that the distributions $F_{i}$ and $F_{0}$ are not required to belong to the same (exponential) parametric family.  Moreover, it can be seen as a discrete approximation to  the continuous setup (\ref{koop}). Such an approximation is necessary in practice, since the continuously weighted likelihood ratio (\ref{cont}) is not usually implementable without such a discretization.  

However, the main motivation for the discrete setup is that it arises naturally in many applications. Consider, for example, the so-called $L$-sample \textit{slippage problem}, where there are $L$ sources of observations (``channels'' or ``populations'')  and there are \textit{two} possibilities for the distribution of each source (\textit{in} and \textit{out of} control). This problem has a variety of important applications, in particular in cybersecurity (see ~\citet{Tartakovskyetal-SM06, Tartakovskyetal-IEEESP06}) and in target detection (see \citet{TarVeerASM2004, Tartakovskyetal-IEEEIT03}). 

Our main contribution in the discrete setup is that we find a mixing distribution $\{p^{0}_{i}\}$ which makes the corresponding mixture test  \textit{nearly minimax} in the sense that it attains $\inf_{T \in \cca} \max_{i} \, (I_{i} \, \Exp_{i}[T])$ up to an $o(1)$ term as $\alpha \rightarrow 0$, where $I_{i}$ is the Kullback--Leibler distance between $F_{i}$ and $F_{0}$ (see Theorem \ref{theo2}). The main components  of the proof are finding a nearly Bayes rule for a decision problem with non-homogeneous sampling costs in $\cP$ and obtaining a high-order
asymptotic expansion for  $\Exp_{i}[T_{A}]$ up to an $o(1)$ term as well as an asymptotic approximation for the ``error probability'' $\Pro_{0}(T_{A}<\infty)$ as $A \rightarrow \infty$.  

\subsection{Misspecification and the Appropriate Minimax Criterion}

As we will see, the expansion for $\Exp_{i}[T_{A}]$ remains valid even when $p_{i}=0$, as long as certain additional conditions are satisfied  (see (\ref{theo1})). That is, we allow the number of active components, $\tilde{K}= \# \{p_{i}: p_{i}\neq 0\}$, of an arbitrary mixture rule to be smaller than $K$. It is useful to incorporate this case in our analysis, since the ``true'' distribution may not be included in $\cP$. 
For example, in the  slippage problem, the actual number of out-of-control channels is typically not known in advance. Thus, 
the cardinality of $\cP$ is $K= \sum_{l=1}^{L} \binom{L}{l}=2^{K}-1$. However, if a designer  assumes that only one channel can be out-of-control, which is the hardest case to detect, the resulting mixture rule will assign a positive weight to only $L$ of the $K$ probability measures in $\cP$, so that 
$\tilde{K}=L<K$. Another case where such a misspecification arises naturally is when approximating a continuous alternative hypothesis with a discrete set of points. Then, it is useful to evaluate the performance of the discrete mixture rule also \textit{between} the points that were used for its design.

Finally, allowing some components of the mixing distribution to be 0 helps to explain why we 
chose to design a sequential test that attains asymptotically $\inf_{T \in \cca} \max_{i} \, (I_{i} \Exp_{i}[T])$  instead of $\inf_{T \in \cca} \max_{i} \, \Exp_{i}[T]$,  which would be the straightforward minimax criterion. Indeed, in Subsection~\ref{ss:inef} we will see that when the Kullback--Leibler numbers $\{I_{i}\}$ are not identical, the latter criterion cannot be attained asymptotically, not even up to a first order, by a mixture rule that gives positive weights to all of its components. Thus,  minimizing the maximal expected sample size is an inappropriate criterion, since it dictates the use of a sequential test, $T^{*}$, that will not even be uniformly first-order asymptotically optimal, i.e., the ratio $\Exp_{i}[T^{*}]$/ $\inf_{T \in \cca} \Exp_{i}[T]$  will not converge to 1 as $\alpha \rightarrow 0$ for every  $1 \leq i\leq K$.

On the other hand,  the criterion $\inf_{T \in \cca} \max_{i} \, (I_{i} \Exp_{i}[T])$ leads to a non-trivial mixture test with $p_{i}>0$ for every $1\leq i \leq K$,  which (just like \textit{any} other fully-supported mixture rule)  attains  $\inf_{T \in \cca} \Exp_{i}[T]$ as $\alpha \rightarrow 0$ up to 
a constant for every $1 \leq i  \leq K$. Moreover, it is a natural minimax criterion  since, as we already mentioned above, $\max_{i} (I_{i} \Exp_{i}[T])= \max_{i} \Exp_{i}[\log \Lambda_{T}^{i}]$ is the maximum Kullback--Leibler distance between $\cP$ and $\Pro_{0}$ based on the observations up to time $T$. Thus, this criterion  provides a natural and meaningful way to express the minimax property and select a particular mixture rule for our  problem.

\subsection{Anscombe's Condition and Nonlinear Renewal Theory}

We would like at this point to highlight the connection of our work with the celebrated paper of  \citet{Ans1952}, where he insightfully
introduced the notion of  uniform continuity in probability and showed that it constitutes a sufficient condition for preserving convergence in distribution when using random times. More specifically, Anscombe called a sequence $\{\xi_{n}\}$ \textit{uniformly continuous in probability} (u.c.i.p), if  for every $\varepsilon>0$ there exists a $\delta>0$ such that  
\begin{equation}\label{ucip}
\Pro\brc{\max_{0 \leq  k \leq n \delta }|\xi_{n+k}-\xi_{n}| \geq \varepsilon} < \varepsilon \quad \text{for every}~~n\in \mathbb{N}.
\end{equation}
Moreover, he  proved that if a \textit{u.c.i.p.} sequence $\{\xi_{n}\}$ converges in distribution to a random variable $\xi$ as $n \rightarrow \infty$
and $\{t_{c}\}$ is a family of positive integer-valued random variables such that $t_{c}/c$ converges  in probability as $c \rightarrow \infty$ to a finite limit, 
then  $\{\xi_{t_{c}}\}$ also converges to $\xi$ as $c \rightarrow \infty$. This theorem has had a profound impact on the field of Sequential Analysis, 
since it provided the basis for developing Central Limit Theorems (CLTs) for stopped random walks and families of stopping times. However, the notion of 
uniform continuity in probability plays an important role in a much wider range of sequential problems, including the one we consider in this paper. 
The reason is its deep connection with \textit{nonlinear renewal theory}, which is the main tool that we use in order to describe the asymptotic performance of mixture rules. The corresponding analysis for continuous mixture rules was done by \citet{PollakSiegmund-AS75}
who first used such ideas before a general theory was presented by  \citet{LaiSieg-77,LaiSieg-79}.

More specifically, assuming that $p_{i}>0$, we can decompose the logarithm of the mixture statistic (\ref{disc})  as 
$\log \Lambda_n=\log \Lambda_n^{i}+Y_n^{i}$, where $Y_n^{i}$ is defined in (\ref{Y}) below. The idea then is that the asymptotic distribution of the overshoot $\log (\Lambda_{T_{A}}/A)$ as $A \rightarrow \infty$ will be the same as if $Y_n^{i}$ was 0, as long as $Y_n^{i}$, $n=1,2,\dots$ are ``slowly changing'' compared to the $\Pro_{i}$-random walk $\{\log \Lambda_n^{i}\}$. This observation leads to an accurate  approximation for $\Pro_{0}(T_{A}<\infty)$, and it is also the basis for the high-order expansion of $\Exp_{i}[T_{A}]$ (for which additional integrability and convergence conditions on $Y_n^{i}$  are required).

Nonlinear renewal theory makes the above argument rigorous by formalizing the notion of a ``slowly changing'' sequence. Specifically, $\{\xi_{n}\}$ is  said to be \textit{slowly changing}, if it is uniformly continuous in probability and satisfies the probabilistic growth condition 
\begin{equation} \label{gro}
\max_{0 \leq  k \leq n } |\xi_{k}| =o_{p}(n) \quad \text{as}~~ n \to \infty,
\end{equation}
i.e., $n^{-1} \max_{0 \leq  k \leq n } |\xi_{k}| \to 0$ in probability. Therefore, uniform continuity in probability is at the core of nonlinear renewal theory being the key condition that allows us to understand the behavior of overshoots of perturbed random walks, and consequently, a variety of ``sequential objects'', such as the mixture-based sequential tests  that we consider in this paper.  

Finally, we should note that using Anscombe's theorem we can establish the asymptotic normality of the (standardized)  mixture stopping rules $\{T_{A}\}$ as  $A \rightarrow \infty$. Whereas we do not need this property for our purposes, it is useful since it justifies using the expectation of $T_{A}$ in order to quantify its performance. 

\subsection{Organization of the Paper}

The rest of the paper is organized as follows. In Section~\ref{s:DMR}, we focus on discrete mixture rules and study their asymptotic  performance and optimality properties. In Section~\ref{s:Expfamily}, we consider the  case of an exponential family with continuous parameter.  
Section~\ref{s:MC} illustrates our findings with simulation experiments in the normal case. 
In Section~\ref{s:Rami}, we discuss ramifications  of our work in testing of two hypotheses and in sequential change detection, and we conclude in Section \ref{s:Conclu}.

\section{DISCRETE MIXTURE RULES}\label{s:DMR}

In this section we assume that  $\cP=\{\Pro_{i}\}_{i=1,\ldots, K}$ and we let  $\{p_{i}\}$ be an arbitrary probability mass function, i.e.,
$p_{i} \geq 0$ for every $i=1,\ldots, K$ and $\sum_{i=1}^{K} p_{i}=1$.

\subsection{Notation and Assumptions}\label{ss:Notation}

Let $\Lambda_n$ be as defined in \eqref{disc} and let $Z_{n}=\log \Lambda_{n}$. Then the mixture rule  (\ref{rule})  calls for stopping and accepting the hypothesis $\Hyp: \Pro \in \cP$ (rejecting the null hypothesis $\Hyp_0: \Pro=\Pro_0$)  at  
\begin{equation} \label{rule2}
T_{A} = \inf \{n \geq 1: Z_{n} \geq \log A \},
\end{equation}
where $T_A=\infty$ if there is no such $n$. For every $i=1,\ldots,K$, we set
\begin{equation} \label{not1}
\Lambda_{n}^{i} = \prod_{m=1}^{n} \frac{dF_{i}(X_m)}{dF_{0}(X_m)} \quad  \text{and} \quad  Z_{n}^{i} =\log \Lambda_{n}^{i}= 
\sum_{m=1}^{n} \log \frac{dF_{i}(X_m)}{dF_{0}(X_m)}, \quad  n \in \mathbb{N},
\end{equation}
and we  define the one-sided SPRTs
\begin{equation} \label{sprt}
T^{i}_{A}= \inf \{n \geq 1: \Lambda_{n}^{i} \geq A \}=\inf \{n \geq 1: Z_{n}^{i} \geq \log A \},
\end{equation}
where $A>1$ is a fixed threshold.

For every $i,j=1, \ldots,K$, we assume that  $0< \Exp_{j}|Z_{1}^{i}| < \infty$, where $\Exp_{j}[\cdot]$ refers to expectation with respect to $\Pro_{j}$, and we set
\begin{equation} \label{KL}
I_{i} = \Exp_{i}[Z_{1}^{i}] \qquad \text{and} \qquad  I_{ji} = \Exp_{j}[Z_{1}^{j}-Z_{1}^{i}]= I_{j} - \Exp_{j}[Z_{1}^{i}],
\end{equation}
i.e., $I_{j}$  ($I_{ji}$) is the  Kullback--Leibler divergence of $F_{j}$ versus $F_{0}$ ($F_{i}$). Therefore,
$\{Z_{n}^{i}\}_{n \ge 1}$ is a  random walk under $\Pro_{j}$ whose increments have mean $\Exp_{j}[Z_{1}^{i}]= I_{j}- I_{ji}$.
If $\Exp_{j}[Z_{1}^{i}]>0$, or equivalently $I_{j}> I_{ji}$, then, by renewal theory,
the asymptotic distribution of the overshoot $\eta_{A}^{i}=Z^{i}_{T^{i}_{A}} -\log A$ under $\Pro_{j}$ is well-defined and we denote it as 
$$
\cH_{j|i}(x) = \lim_{A \rightarrow \infty} \Pro_{j}(\eta_{A}^{i} \leq x).
$$
More specifically, $\cH_{j|i}$  can be defined in terms of the ladder variables of the $\Pro_{j}$-random walk $\{Z_{n}^{i}\}$. For the sake of brevity, we write $\cH_{i}= \cH_{i|i}$ for the asymptotic distribution of $\eta_{1}^{i}$ under $\Pro_{i}$, which is always well-defined since $\Exp_{i}[Z_{1}^{i}]=I_{i}>0$.

With a change of measure $\Pro_0 \mapsto \Pro_i$  it can be easily shown that 
\begin{equation} \label{er}
A \, \Pro_{0}(T^{i}_{A} < \infty) = A \, \Exp_i \brcs{1/\Lambda^i_{T^{i}_{A}} \ind{T^{i}_{A} < \infty}} =  \Exp_i \brcs{\exp (-\eta_{A}^{i}) \ind{T^{i}_{A} < \infty}}  \rightarrow \delta_{i}  \quad \text{as}~ A \rightarrow \infty,
\end{equation}
where $\delta_{i}$ is the Laplace transform of $\cH_{i}$, i.e.,
\begin{equation} \label{delta}
\delta_{i}= \int_{0}^{\infty} e^{-x} \, \cH_{i}(dx)= \lim_{A \rightarrow \infty} \Exp_{i}[ e^{- \eta_{A}^{i}}].
\end{equation}
Note that the quantity $\delta_{i}$ is also very important when designing the one-sided test  $T_{A}^{i}$. More specifically, \citet{Lorden-AS77} showed that if $c$ is the cost of every observation, then the one-sided SPRT $T_{A}^{i}$ with $A=\delta_{i} I_i/c$  attains $\inf_{T} [\Pro_{0}(T<\infty)+c\Exp_{i}[T]]$, where the infimum is taken over all stopping times. 

If $\Exp_{j}[\max\{0,Z_{1}^{i}\}^{2}]<\infty$, then from Wald's identity, (\ref{KL}) and renewal theory (\citet[Corollary 2.2]{Woodroofe-book82}), 
we have
\begin{equation} \label{renew}
[I_{j}-I_{ji}] \; \Exp_{j}[T_{A}^{i}] = \log A + \varkappa_{j|i} + o(1) \quad \text{as}~~ A \to \infty ,
\end{equation}
where $\varkappa_{j|i}$ is the average of $\cH_{j|i}$, i.e.,
\begin{equation} \label{kappa}
\varkappa_{j|i}= \int_{0}^{\infty} x \,  \cH_{j|i}(dx)= \lim_{A \rightarrow \infty}  \Exp_{j}[ \eta_{A}^{i}].
\end{equation}
It is a direct consequence of (\ref{renew}) that
\begin{equation} \label{renew2}
I_{i} \; \Exp_{i}[T_{A}^{i}] = \log A + \varkappa_{i} + o(1)  \quad \text{as}~~ A \to \infty,
\end{equation}
where $\varkappa_{i}= \varkappa_{i|i}$. In the next section, we show that the limiting average overshoots $\varkappa_{1}, \dots,\varkappa_K$ completely determine the  (optimal) mixing distribution of the nearly minimax mixture rule.

If $\Pro_{0}(T^{i}_{A}<\infty)=\alpha$, where $\alpha$ is a predefined number ($0<\alpha<1$), then (\ref{er}) and (\ref{renew2}) imply that
\begin{equation} \label{per1}
I_{i} \,  \Exp_{i}[T_{A}^{i}] = |\log \alpha| +  \log (\delta_{i} \, e^{\varkappa_{i}}) +o(1) \quad \text{as}~~ \alpha \to 0.
\end{equation}
Due to (\ref{optimum}), this is the  optimal asymptotic performance under $\Pro_{i}$ up  to an $o(1)$ term. Therefore, asymptotic approximation \eqref{per1} provides a benchmark for the performance of any stopping time under $\Pro_{i}$.

In order to study the performance of $T_{A}$ under $\Pro_{i}$ even if $p_{i}=0$,  for every $i=1,\ldots,K$ we define the index
\begin{equation} \label{star}
i^{*}= \arg \max_{j: p_{j}>0} \Exp_{i}[Z_{1}^{j}]= \arg \min_{j:p_{j}>0} I_{ij}
\end{equation}
and we assume that it is unique. When $p_{i}>0$, this is obviously the case since $i^{*}=i$. On the other hand, when $p_{i}=0$, $i^{*}$ represents the ``active'' index that is closest to $i$, in the sense of the Kullback--Leibler distance for the corresponding distributions. 
Thus, assuming that $i^{*}$ is unique, we exclude the case that there are two or more active indexes that are ``equidistant'' from $i$  when $p_{i}=0$. Then, for every $i=1,\ldots, K$, we have the decomposition $Z_{n}=Z_{n}^{i^{*}}+ Y_{n}^{i^{*}}$, where
\begin{equation}  \label{Y}
Y_{n}^{i^{*}}= \log p_{i^{*}} + \log \left( 1+ \sum_{j \neq i^*}  \frac{p_j}{p_{i^{*}}} \,  \frac{\Lambda_{n}^{j}}{\Lambda_{n}^{i^{*}}} \right), \quad n \in \mathbb{N}.
\end{equation}
Based on this decomposition and the fact that when $i^*$ is unique the sequence $\{Y_{n}^{i^{*}}\}$ is slowly changing, we are able to use nonlinear renewal theory and understand the  asymptotic behavior of the mixture rule $T_{A}$. When $p_{i}=0$ and $i^{*}$ is not unique, this decomposition is not valid and this case has to be considered separately.  We do not consider this case here, since this would break the flow of the presentation without adding any insight to our main points. Methods similar to those developed in 
\citet{dragalin-it00} and \cite{Tartakovskyetal-IEEEIT03} can be used for this purpose.

Finally, in the case $p_{i}=0$, we will also need the following Cram\'er-type condition:
\begin{condition} \label{cond}
For every $j \neq i^{*}$ with $p_{j}>0$ there exists $\gamma_{j}>0$ such that $g_{j}(\gamma_{j})=1$ and $g_{j}'(\gamma_{j})<\infty$,
where $g_{j}(t)= \Exp_{j}[e^{t(Z_{1}^{j}-Z_{1}^{i^{*}})}].$
\end{condition}

\subsection{Modes of Asymptotic Optimality}\label{ss:Modes of asympt optim}
Ideally,  we would like to find an optimal test $T_{{\rm opt}} \in \cca$ that minimizes the expected sample size $\inf_{T\in \cca} \Exp_{i}[T]$ for all $i=1,\dots,K$,  where $\cca=\{T:\Pro_{0}(T<\infty)\leq \alpha\}$. Since this is an extremely difficult task (if at all possible), we would like to find a test $T_{\rm o} \in \cca$ that  attains  $\inf_{T\in \cca} \Exp_{i}[T]$ at least \textit{asymptotically} for all  $i=1,\ldots, K$. We distinguish between the following three notions of asymptotic optimality. We  say that $T_{\rm o}$  minimizes $\inf_{T\in \cca} \Exp_{i}[T]$ to \textit{first}-order if $\Exp_{i}[T_{\rm o}]= \inf_{T\in \cca} \Exp_{i}[T] \, (1+o(1))$;  to \textit{second}-order  if $\Exp_{i}[T_{\rm o}]= \inf_{T\in \cca} \Exp_{i}[T] + O(1)$; and to \textit{third}-order, if $\Exp_{i}[T_{\rm o}]= \inf_{T\in \cca} \Exp_{i}[T] +o(1)$, where $O(1)$ is asymptotically bounded and $o(1)$ an asymptotically vanishing term as $\alpha \rightarrow 0$. 

Since the one-sided SPRT $T_A^i$ is exactly optimal under $\Pro_{i}$, it follows from (\ref{per1}) that
\[
\inf_{T\in \cca} \Exp_{i}[T] = \frac{1}{I_i} \brcs{|\log \alpha|  +\log (\delta_{i} \, e^{\varkappa_{i}})} +o(1) \quad \text{as}~~ \alpha \to 0.
\]
Using this fact along with Theorem~\ref{theo1},
we will see that a mixture rule is second-order asymptotically optimal under every $\Pro_{i} \in \cP$ if and only if
it assigns positive weights to all  probability measures in the alternative hypothesis, that is $p_{i}>0$ for every $i=1,\ldots,K$. 
In other words, for every fully-supported mixture test $T_{A}$ with $\Pro_{0}(T_{A}<\infty)=\alpha$,
the expectation $\Exp_{i}[T_{A}]$  has a bounded distance from $\inf_{T\in \cca} \Exp_{i}[T]$ as $A \rightarrow \infty$  for every $i=1,\ldots, K$.

\subsection{Asymptotic Performance}\label{ss:Asympt perf}

The main result of this subsection is Theorem~\ref{theo1}, which provides a high-order asymptotic approximation for $\Exp_{i}[T_{A}]$ as $A \rightarrow \infty$. Its proof is based on Lemmas~\ref{leo1}--\ref{leo4}. In Lemma \ref{leo1} we present the main properties of  the sequence $\{Y_{n}^{i^{*}}\},$ in Lemma \ref{leo2} we obtain sufficient conditions for $T_{A}$ to have power 1 under $\Pro_{i}$, and in Lemmas  \ref{leo3} and \ref{leo4} we obtain asymptotic approximations for  $\log \Pro_{0}(T_{A}<\infty)$ and  $\Exp_{i}[T_{A}]$ in terms of the threshold $A$.

\begin{lemma} \label{leo1}
For every $i$, $\Pro_{i}(Y^{i^{*}}_{n} \downarrow \log p_{i^{*}})=1$, and hence the sequence $\{Y^{i^{*}}_n\}$ is slowly changing under $\Pro_{i}$. Moreover, 
if either $p_{i}>0$ or if $p_{i}=0$ and Condition \ref{cond} is satisfied, then there exists $\gamma_{i^{*}} >0$  such that the following asymptotic equality holds
\begin{equation} \label{asy}
\Pro_{i} \Bigl( \max_{0 \leq k \leq n}  |Y_{k}^{i^{*}} -\log p_{i^{*}}| > x \Bigr) =O\left(e^{-\gamma_{i^{*}} x}\right) \quad \text{as}~~ x \rightarrow \infty.
\end{equation}
\end{lemma}

\begin{proof}
From (\ref{Y}) it follows directly that $Y_{n}^{i^{*}} \geq \log p_{i^{*}}$. Moreover, by the strong law of large numbers, 
$$ 
\frac{1}{n} \log \, \frac{\Lambda_{n}^{j}}{\Lambda_{n}^{i^{*}}} = \frac{ Z_{n}^{j}-Z_{n}^{i^{*}}}{n} \xrightarrow[n\to\infty]{\Pro_{i}-\text{a.s.}}
\Exp_{i}[Z_{1}^{j}-Z_{1}^{i^{*}}]= I_{i i^{*}} - I_{ij} \quad \text{for every}~ j\neq i^{*}.
$$
Since $ I_{i i^{*}} <  I_{i j}$ (by the definition of $i^{*}$), it follows that $\Pro_{i}( \Lambda_{n}^{j} / \Lambda_{n}^{i^{*}}  \rightarrow 0)=1$ for every $j \neq  i^{*}$ with $p_{j}>0$, and consequently, $\Pro_{i}( Y^{i^{*}}_{n} \rightarrow \log p_{i^{*}})=1$. As a result, $\{Y^{i^{*}}_n\}$ satisfies  \eqref{ucip} and (\ref{gro}). Thus, it is a slowly changing sequence under $\Pro_{i}$.

To prove (\ref{asy}), suppose first that $p_{i}>0$. Then, $i^{*}=i$ and $\sum_{j \neq i} \Lambda^{j}_n/\Lambda^{i}_n$ is a $\Pro_{i}$-martingale with mean $K-1$. Thus, from (\ref{Y}) and  Doob's submartingale inequality we obtain
\begin{equation} \label{probformaxY}
\begin{aligned}
\Pro_{i}\brc{\max_{0 \leq k \leq n} |Y_{k}^{i} -\log p_{i}|> x} &= 
\Pro_{i} \brc{1+ \max_{0 \leq k \leq n}  \sum_{j \neq i} \frac{p_{j}}{p_{i}} \, \frac{\Lambda_{k}^{j}}{\Lambda_{k}^{i}} > e^{x} } \nonumber \\
 &\leq   \Pro_{i} \brc{\max_{0 \leq k \leq n}  \sum_{j \neq i}  \frac{\Lambda_{k}^{j}}{\Lambda_{k}^{i}} >  p_{i} (e^{x} -1)}  \leq  \frac{(K-1)}{p_{i} (e^{x} - 1)},
 \end{aligned}
\end{equation}
which implies that (\ref{asy}) holds with $\gamma_{i}=1$.

Suppose now that $p_{i}=0$, in which case $i^{*} \neq i$.  Then, working as in (\ref{probformaxY}) and using the following inclusion
$$ 
\Bigl\{ \max_{0 \leq k \leq n}   \sum_{j \neq i^{*}: p_{j}>0}  \, \Lambda_{k}^{j} / \Lambda_{k}^{i^{*}}  >  y \Bigr\}
 \subset \bigcup_{j \neq i^{*}: p_{j}>0} \Bigl\{ \max_{0 \leq k \leq n}  \Lambda_{k}^{j} / \Lambda_{k}^{i^{*}}  >  \frac{y}{K-1} \Bigr\},
 $$
which holds for every  positive constant $y$, we obtain
\begin{equation*}
\begin{aligned}
\Pro_{i}\brc{ \max_{0 \leq k \leq n} |Y_{k}^{i^{*}} -\log p_{i^{*}}| > x} &\leq
\Pro_{i} \brc{ \max_{0 \leq k \leq n}   \sum_{j \neq i^{*}: p_{j}>0}  \, \frac{\Lambda_{k}^{j}}{\Lambda_{k}^{i^{*}}}  >  p_{i^{*}} (e^{x}-1)}
\\
&\leq  \sum_{j \neq  i^{*}} \, \Pro_{i} \brc{ \max_{0 \leq k \leq n} \frac{\Lambda_{k}^{j}}{ \Lambda_{k}^{i^{*}}}  >  \frac{p_{i^{*}} (e^{x}-1)}{K-1}}
  \\
  &=  \sum_{j \neq i^{*}} \Pro_{i} \brc{\max_{0 \leq k \leq n} [Z_{k}^{j}- Z_{k}^{i^{*}}]  >  x+ \Theta(1)},
  \end{aligned}
 \end{equation*}
where $\Theta(1)$ is a term that is asymptotically bounded from above and from below as $x  \rightarrow \infty$.
For every $j \neq i^{*}$, the process $\{Z_{n}^{j}-Z_{n}^{i^{*}}\}_{n\ge 1}$ is a $\Pro_{i}$-random walk whose increments have mean $\Exp_{i}[Z_{1}^{j}-Z_{1}^{i^{*}}]<0,$
which is negative due to the definition of $i^{*}$. Thus, by  Condition \ref{cond},  for every $j \neq i^{*}$ with $p_{j}>0$ there exists a positive constant $\gamma_{j}>0$ such that  
$$
\Pro_{i}\brc{\max_{0 \leq k \leq n} [Z_{k}^{j}- Z_{k}^{i^{*}}]  >  x+ \Theta(1)}=O\brc{e^{-\gamma_{j} x}},
$$
which implies that  (\ref{asy})  is satisfied with $\gamma_{i^{*}}=\min\{\gamma_{j}: j \neq i^{*}, p_{j}>0\}$.
\end{proof}

\begin{lemma} \label{leo2}
If either $p_{i}>0$ or $p_{i}=0$ but $I_{i} > I_{ii^*}$, then $\Pro_{i}(T_{A} < \infty)=1 \;$ $\forall \; A>1$  and  $\Pro_{i}(T_{A} \rightarrow \infty)=1$ as $A \rightarrow \infty$.
\end{lemma}

\begin{proof}
First of all, we observe that $\{Z_{n}^{i^{*}}\}$ is a $\Pro_{i}$-random walk whose increments have mean $\Exp_{i}[Z_{1}^{i^{*}}]= I_{i}- I_{ii^*}$.
Due to the assumption of the lemma, the latter is positive, and therefore, $\Pro_{i}(Z^{i^{*}}_{n} \rightarrow \infty)=1$ as $n \to \infty$. Since 
\begin{equation} \label{deco}
T_{A} = \inf \{n \geq 1:  Z_{n}^{i^{*}}  + Y_{n}^{i^{*}} \geq \log A \},
\end{equation}
and, by Lemma~\ref{leo1},  $\Pro_{i}(Y^{i^{*}}_{n} \downarrow \log p_{i^{*}})=1$  we conclude that $T_{A}$ terminates $\Pro_{i}$-a.s.\ and that  $\Pro_{i}(T_{A} \rightarrow \infty)=1$ as $A \rightarrow \infty$.
\end{proof}

\begin{lemma} \label{leo3}
For every $A>1$, $T_{A}$ is a test of level $1/A$, i.e., $\Pro_{0}(T_{A}< \infty) \leq 1/A$. Moreover, if for every $i$ such that $p_{i}>0$
the distribution of  $Z_{1}^{i}$ is non-arithmetic, then
\begin{equation} \label{aaaa}
A \,  \Pro_{0}(T_{A}< \infty)  \rightarrow \sum_{i: p_{i}>0}    p_{i} \, \delta_{i} \quad \text{as}~~ A\to\infty.
\end{equation}
\end{lemma}

\begin{proof}

Define the probability measure $\Pro= \sum_{i: p_{i}>0} p_i \, \Pro_{i}$. If $p_{i}>0$,  then by Lemma \ref{leo2},
$\Pro_{i}(T_{A}< \infty)=1$, and therefore,  $\Pro(T_{A}< \infty)=1$. Moreover,
\begin{equation}  \label{proba}
\frac{d \Pro}{d\Pro_{0}} \Big|_{\Fc_{n}} =\Lambda_{n} = \sum_{i=1}^{K}  p_i \Lambda_{n}^{i}.
\end{equation}
Therefore, if $\Exp[\cdot]$ denotes expectation with respect to $\Pro$, change of  measure $\Pro_0 \mapsto \Pro$ yields
\begin{equation} \label{como}
 A \, \Pro_{0}(T_{A}< \infty) = A \, \Exp[e^{-Z_{T_{A}}}] =   \Exp\brcs{e^{-(Z_{T_{A}}- \log A)}} \leq 1,
\end{equation}
which proves the first assertion. Furthermore, from (\ref{como}) and the definition of $\Pro$ we have
\begin{equation} \label{com2}
 A \, \Pro_{0}(T_{A}< \infty) =\sum_{i: p_{i}>0}  p_i  \, \Exp_{i}\brcs{e^{-(Z_{T_{A}}- \log A)}}.
\end{equation}
If $p_{i}>0$, then $i^{*}=i$  and we have  the decomposition $Z_{n}= Z_{n}^{i}+Y_{n}^{i}$, where $\{Z_{n}^{i}\}$ is a $\Pro_{i}$-random walk with positive mean $I_{i}$  and  $\{Y_{n}^{i}\}$ is a slowly changing sequence under $\Pro_{i}$. Therefore,
if also the distribution of $Z_{1}^{i}$ is non-arithmetic, then $Z_{T_{A}}- \log A$ converges weakly as $A \rightarrow \infty$ to $\cH_{i}(\cdot)$ under $\Pro_{i}$ (see \citet[Theorem 4.1]{Woodroofe-book82}). Thus, recalling the definition of $\delta_{i}$ in (\ref{delta}) and applying the Bounded Convergence Theorem, from (\ref{com2}) we obtain (\ref{aaaa}). This completes the proof.

\end{proof}

\begin{lemma} \label{leo4}
Suppose that $Z_{1}^{i^{*}}$ has a non-arithmetic distribution with a finite second moment under $\Pro_{i}$. If either $p_{i}>0$ or $I_{i} > I_{ii^*}$ and Condition \ref{cond} holds, then 
\begin{equation} \label{perf10}
(I_{i}- I_{ii^*}) \; \Exp_{i}[T_{A}] = \log A + \varkappa_{i|i^{*}} - \log p_{i^{*}} + o(1) \quad \text{as}~~ A \to \infty.
\end{equation}
\end{lemma}

\begin{proof}
Write $D_{i i^*}= I_{i}- I_{ii^*}$. Since $\{Z_{n}^{i^{*}}\}_{n\ge 1}$ is a $\Pro_{i}$-random walk whose increments have non-arithmetic distribution and positive mean $\Exp_{i}[Z_{1}^{i^{*}}]=D_{i i^*}$, asymptotic approximation (\ref{perf10}) follows from Woodroofe's nonlinear renewal theorem (see Theorem 4.5 in \citet{Woodroofe-book82}), as long as the the following conditions are satisfied:
\begin{enumerate}
\item[(A1)] $\{\max_{0 \leq k \leq n} |Y^{i^{*}}_{k+n} - \log p_{i^{*}}| \}_{n\ge 1}$ is a uniformly integrable sequence;
\item[(A2)] $\sum_{n=0}^{\infty} \Pro_{i}(|Y_{n}^{i^{*}}- \log p_{i^{*}}| \leq -n \varepsilon) < \infty$ for some $\varepsilon \in (0, D_{i i^*})$;
\item[(A3)] $\{Y^{i^{*}}_{n} -\log p_{i^{*}} \}_{n \ge 1}$ converges in distribution;
\item[(A4)]  $\Pro_{i}(T_{A} \leq N_{A}) =o(1/N_{A})$ as $A \rightarrow \infty$ for some $\varepsilon > 0$, where $N_{A}=\lfloor (\varepsilon \, \log A)/D_{i i^*} \rfloor$.
\end{enumerate}
Condition (A1) is satisfied because $\sup_{n} |Y_{n}^{i^{*}} - \log p_{i^{*}}| $ is $\Pro_{i}$-integrable. Indeed, from (\ref{asy}), which holds
if either $p_{i}>0$ or Condition \ref{cond} holds (see Lemma \ref{leo1}), we have
$$
\Exp_{i}\brcs{\sup_{n} |Y_{n}^{i^{*}} - \log p_{i^{*}}|}= \lim_{n} \int_{0}^{\infty}  \Pro_{i}\brc{\max_{0 \leq k \leq n} (Y_{k}^{i^{*}}-\log p_{i^{*}}) > x} \, dx<\infty.
$$
Condition (A2) is clearly  satisfied, since $Y_{n}^{i^{*}} \geq \log  p_{i^{*}}$ for every $n$, whereas condition (A3) is also satisfied, since 
$\{Y_{n}^{i^{*}} -\log p_{i^{*}}\}$ converges to 0 $\Pro_{i}$-a.s.

In order to verify  (A4), we start with the following inclusion, which holds for every $n \in \mathbb{N}$ and $x>0$,
$$
\set{\max_{0 \leq k \leq n} Z_{k} > x} \subset \set{\max_{0 \leq k \leq n } Z_{k}^{i^{*}} > x/2} \bigcup \set{\max_{0 \leq k \leq n }  Y_{k}^{i^{*}} > x/2 } 
$$
and which implies that
\begin{align*}
\Pro_{i}(T_{A} \leq N_{A}) &\leq
\Pro_{i}\brc{ \max_{0 \leq k \leq N_{A}} Z_{k}   \geq  \log A } 
\\
& \leq  \Pro_{i}\brc{ \max_{0 \leq k \leq N_{A}} Z^{i^{*}}_{k}   \geq  \log \sqrt{A} }
+ \Pro_{i}\brc{ \max_{0 \leq k \leq N_{A}}  Y_{k}^{i^{*}}  \geq  \log \sqrt{A} } .
\end{align*}
Therefore, it suffices to show that both terms on the right-hand side are of order $o(1/\log A)$ as $A \rightarrow \infty$.

Consider the second term. If $p_i > 0$, then $i^*=i$ and, by \eqref{probformaxY}, 
\[
 \Pro_{i}\brc{ \max_{0 \leq k \leq N_{A}} | Y_{k}^{i} -\log p_{i} | \geq  \log \sqrt{A} } \le  \frac{K-1}{ p_{i}(\sqrt{A} -1)} = O(A^{-1/2}).
\]
Now, if $p_i=0$ and Condition~\ref{cond} is satisfied, then by \eqref{asy} 
\[
 \Pro_{i}\brc{ \max_{0 \leq k \leq N_{A}} | Y_{k}^{i^*} -\log p_{i^{*}} | \geq  \log \sqrt{A} } = O(A^{-\gamma_{i^{*}}/2}) \quad \text{as}~~ A \to \infty.
\]

Finally, consider the first term. We have
\[
\begin{aligned}
\Pro_{i}\brc{ \max_{0 \leq k \leq N_{A}} Z_{k}^{i^*}   \geq  \log \sqrt{A} } & = \Pro_{i}\set{ \max_{0 \leq k \leq N_{A}} \brc{Z_{k}^{i^*} - D_{i i^*} N_A}   \geq  \frac{1}{2} \log A - D_{i i^*} N_A} 
\\
&= \Pro_{i}\set{ \max_{0 \leq k \leq N_{A}} \brc{Z_{k}^{i^*} - D_{i i^*}N_A}   \geq  \frac{1-2\varepsilon}{2\varepsilon} D_{i i^*} N_A}
\\
 &\le  \Pro_{i}\set{ \max_{0 \leq k \leq N_{A}} \brc{Z_{k}^{i^*} - D_{i i^*}k}   \geq \gamma N_A}
\end{aligned}
\]
 for some $\gamma >0$.  Write $S_k = Z_{k}^{i^*} - D_{i i^*}k$ and $\sigma^2=\Exp_i S_1^2$ (which is finite by the conditions of lemma).  Note that $\{S_k\}_{k\ge 1}$ is a zero-mean $\Pro_i$-martingale, so that $\{S_k^2\}_{k\ge 1}$ is a submartingale with respect to $\Pro_i$. Applying Doob's maximal submartingale inequality, we obtain
\[
\begin{aligned}
\Pro_i \brc{\max_{0 \leq k \leq N_{A}} |S_k| \ge \gamma N_A} & \le  \frac{1}{(\gamma N_A)^2} \Exp_i\brcs{S_{N_A}^2 \ind{\max\limits_{0 \leq k \leq N_{A}} S_{k} \ge  \gamma N_A}}
\\
&=\frac{1}{\gamma^2 N_A} \Exp_i\brcs{\brc{\frac{S_{N_A}^2}{N_A}} \ind{\max\limits_{0 \leq k \leq N_{A}} S_{k} \ge  \gamma N_A}} .
\end{aligned}
\]
First, it follows that
\[
\Pro_i\brc{\max_{0 \leq k \leq N_{A}} |S_n| \ge \gamma N_A} \le \frac{\sigma^2}{\gamma^2 N_A}\xrightarrow[N_A\to\infty]{} 0 .
\]
Now, we show that
\[
 \Exp_i\brcs{\brc{\frac{S_{N_A}^2}{N_A}} \ind{\max\limits_{0 \leq k \leq N_{A}} S_{k} \ge  \gamma N_A}}
\xrightarrow[N_A\to\infty]{} 0,
\]
as long as $\Exp_i S_1^2=\sigma^2 < \infty$, which implies that 
\[
\Pro_i\brc{\max\limits_{0 \leq k \leq N_{A}}|S_n| > \gamma N_A} = o(1/N_A) \quad \text{as}~~A\to\infty,
\]
i.e., the desired result. By the Central Limit Theorem, $S^2_{N_A}/(N_A \sigma^2)$ converges as $A\to\infty$ in distribution 
to a standard chi-squared random variable with one degree of freedom,  $\chi^2$. Hence, for any $L<\infty$ we have
\[
\begin{aligned}
\Exp_i\brc{\frac{S_{N_A}^2}{N_A}\ind{\max\limits_{0 \leq k \leq N_{A}} S_k > \gamma N_A}} &=
\Exp_i\brcs{\brc{L\wedge\frac{S_{N_A}^2}{N_A}}\ind{\max\limits_{0 \leq k \leq N_{A}} S_k > \gamma N_A}}
\\
&\quad +\Exp_i \brcs{\brc{\frac{S_{N_A}^2}{N_A}-L\wedge\frac{S_{N_A}^2}{N_A}}\ind{\max\limits_{0 \leq k \leq N_{A}} S_k > \gamma N_A}}
\\
&\le L \Pro_i\brc{\max\limits_{0 \leq k \leq N_{A}} S_k > \gamma N_A }+ \Exp_i \brc{\frac{S_{N_{A}}^2}{N_{A}}-L\wedge\frac{S_{N_{A}}^2}{N_{A}}}
\\
&\le \frac{L\sigma^2}{\varepsilon^2 N_{A}}+\sigma^2-\Exp_i \brc{L\wedge\frac{S_{N_A}^2}{N_A}}
\\
&\xrightarrow[N_A\to\infty]{}\sigma^2-\Exp_i \brc{L\wedge\chi^2\sigma^2} \xrightarrow[L\to\infty]{}\sigma^2(1-1)=0.
\end{aligned}
\]
The proof is complete.
\end{proof}

Now everything is prepared to obtain an asymptotic approximation for the expected sample size up to the negligible term $o(1)$.

\begin{theorem} \label{theo1}
Suppose that $Z_{1}^{i^{*}}$ has a non-arithmetic distribution with a finite second moment under $\Pro_{i}$ and that either $p_{i}>0$
or $p_i=0$ and $I_{i} > I_{ii^*}$ and Condition \ref{cond}  holds.  Then
\begin{equation} \label{perfect}
(I_{i}- I_{ii^*}) \; \Exp_{i}[T_{A}] = |\log \Pro_0(T_A<\infty)| + \log \Bigl(\sum_{i:p_{i}>0}  p_{i} \, \delta_{i} \Bigr) + \varkappa_{i|i^{*}} - \log p_{i^{*}}+o(1) \quad \text{as}~~ A\to \infty.
\end{equation}
\end{theorem}

\begin{proof}
Using (\ref{aaaa}), we obtain 
\begin{equation} \label{a}
\log A = |\log  \Pro_{0}(T_{A}< \infty) | + \log \Bigl(\sum_{i: p_{i}>0}    p_{i} \, \delta_{i} \Bigr) + o(1).
\end{equation}
We can then obtain (\ref{perfect}) combining  (\ref{perf10}) and (\ref{a}).
\end{proof}

\begin{remark} \label{remar}
If the desired error probability $\Pro_0(T_A<\infty)=\alpha$ is fixed in advance, usually  it is not possible to choose the threshold $A=A_\alpha$
so that $T_{A}$ is a test of size $\alpha$, i.e., so that $\Pro_{0}(T_{A}<\infty)$ is  exactly equal to $\alpha$. Nevertheless, if $A = \alpha^{-1} \sum_{i=1}^{K} p_{i}\,  \delta_{i}$, then from (\ref{aaaa}) and \eqref{perfect} we have 
\begin{align}
\Pro_{0}(T_{A} < \infty) & =\alpha (1+o(1)) , \nonumber
\\
(I_{i}- I_{ii^*}) \; \Exp_{i}[T_{A}] & = |\log \alpha| + \log \Bigl(\sum_{i:p_{i}>0}  p_{i} \, \delta_{i} \Bigr) + \varkappa_{i|i^{*}} - \log p_{i^{*}}+o(1) \quad \text{as}~~ \alpha\to 0 . \label{perfectalpha}
\end{align}

\end{remark}

The following corollary specializes Theorem \ref{theo1} in the case that $p_{i}>0$.

\begin{corollary} \label{cor1}
Suppose that $p_{i}>0$ and that $Z_{1}^{i}$ has a non-arithmetic distribution with a finite second moment under $\Pro_{i}$.
If $\Pro_{0}(T_{A}<\infty)=\alpha$, then
\begin{equation} \label{perfect2}
I_{i} \; \Exp_{i}[T_{A}] = |\log \alpha| + \log \Bigl(\sum_{i:p_{i}>0}  p_{i} \, \delta_{i} \Bigr) + \varkappa_{i} - \log p_{i}+o(1) \quad \text{as}~~ \alpha\to 0
\end{equation}
and $T_{A}$ is second-order asymptotically optimal under $\Pro_{i}$, that is,
\begin{equation}
\Exp_{i}[T_{A}]= \inf_{T\in \cca} \Exp_{i}[T]+ O(1) \quad \text{as}~~ \alpha\to 0.
\end{equation}
\end{corollary}

This corollary implies that the performance loss of a mixture rule is bounded as $A \rightarrow \infty$ under every $\Pro_{i} \in \cP$, 
as long as $p_{i}>0$ for every $i=1,\ldots,K$. However, when the number of ``active'' components in the mixing distribution,
$\tilde{K}=\#\{p_{i}: p_{i}>0\}$, is very large, only first-order asymptotic optimality can be attained. This is the content of the following corollary of Theorem \ref{theo1}.

\begin{corollary} \label{cor11}
Suppose that $p_{i}>0$ and that $Z_{1}^{i}$ has a non-arithmetic distribution with a finite second moment under $\Pro_{i}$. If 
$\Pro_{0}(T_{A}<\infty)=\alpha$ and $\tilde{K} \rightarrow \infty$ so that $\log \tilde{K}=o(|\log \alpha|)$,  
then $T_{A}$ is first-order asymptotically optimal under $\Pro_{i}$, i.e.,  $\Exp_{i}[T_{A}]= \inf_{T\in \cca} \Exp_{i}[T] \, (1+o(1))$ as $\alpha\to 0$. 
\end{corollary}

\subsection{A Nearly Minimax Discrete Mixture Rule}\label{ss:TOminimax}

The proof of minimaxity is constructed based on an auxiliary Bayesian approach. The method is ideologically similar to that used by \citet{Lorden-AS77} and goes back to the proof of optimality of Wald's SPRT given by \citet{WaldWolfowitz48}. 

More specifically, consider the following Bayesian problem denoted by $\cB(\pi, \{p_i\},c)$. Let $\pi \in (0,1)$ be the prior probability of the null hypothesis $\Hyp_0: \Pro=\Pro_{0}$, and assume that the losses associated with stopping at time $T$ are 1 if $T<\infty$ and the hypothesis $\Hyp_0$ is true and $(c \cdot I_i) \times T$ if $\Pro_i$ is the true probability measure, where $c>0$ is a fixed constant. 
Therefore, the cost of every observation under $\Pro_{i}$ is proportional to  the difficulty of  discriminating between $F_{i}$ and $F_{0}$ measured by the Kullback--Leibler divergence $I_{i}$. Since the prior probability of the alternative hypothesis $\Hyp_{1}: \Pro \in \cP=\bigcup_{i=1}^K \Pro_{i}$ is $(1-\pi) \sum_{i=1}^K p_{i}=(1-\pi)$, the Bayes (integrated) risk associated with an arbitrary stopping time $T$ is
\begin{equation} \label{bayes}
\cR_{c}(T) =  \pi \,  \Pro_{0}(T < \infty) +  c \, (1-\pi) \,  \sum_{i=1}^{K} p_{i} \; I_{i} \, \Exp_{i}[T] .
\end{equation}
Moreover, for any positive constant $Q$ such that $Qc<\pi$, we consider the mixture rule $T_{A_{Qc}}$, where
\begin{equation} \label{qc}
A_{Qc}= \Bigl(\frac{1-Qc}{Qc} \Bigr) \Big/ \Bigl( \frac{1-\pi}{\pi} \Bigr).
\end{equation}
These stopping times have a  natural Bayesian interpretation.
Indeed, write $\Pro^p=\sum_{i=1}^{K} p_{i} \, \Pro_{i}$ and $\Prop^{\pi} = \pi \, \Pro_{0} + (1-\pi) \, \Pro^p$. Then 
\[
\Prop^{\pi}(\cdot \, | \,  \Hyp_0)= \pi \, \Pro_{0}(\cdot) , \quad \Prop^{\pi}(\cdot  \, | \,  \Hyp)= (1-\pi) \, \Pro^p(\cdot) ,
\]
and the posterior probability of the hypothesis $\Hyp_0$ takes the form
$$
\Pi_{n} = \Prop^{\pi} (\Hyp_0 | \Fc_{n})= \frac{1}{1+ \frac{1-\pi}{\pi} \Lambda_{n}} ,  \quad n \in \mathbb{N}.
$$
Thus, $T_{A_{Qc}}$ is the first time that the posterior probability of the null hypothesis becomes smaller than $Qc$, that is,
\begin{equation} \label{repre}
T_{A_{Qc}}= \inf\{n \geq 1: \Lambda_{n} \geq A_{Qc} \} = \inf\{n \geq 1: \Pi_{n} \leq Qc \} .
\end{equation}

Solution of $\cB(\pi, \{p_i\},c)$ requires minimization of the expected loss \eqref{bayes}. In the following lemma we establish Bayesian optimality of the mixture test $T_{A_{Qc}}$ in the problem $\cB(\pi, \{p_i\},c)$ for sufficiently small $c$.

\begin{lemma} \label{leo5}
For any given $\pi \in (0,1)$ and $Q> 1/e$, there exists  $c^{*}$ such that 
\[
\cR_{c}(T_{A_{Qc}})= \inf_{T} \cR_{c}(T) \quad \text{for every}~~c< \pi c^{*},
\]
where infimum is taken over all stopping times.
\end{lemma}

The proof of Lemma \ref{leo5} is methodologically similar to  the proof of Lemma 2 in \citet{Pollak-AS78} (see also \citet{Lorden-AS67}) 
and is presented in the Appendix. This lemma provides the basis for the following important theorem,  which shows that a particular mixing distribution leads to a mixture rule that is  \textit{almost minimax} in the sense of minimizing the Kullback--Leibler information in the worst-case scenario up to an $o(1)$ term.

\begin{theorem} \label{theo2}
Let $\cca=\{T: \Pro_0 (T< \infty) \le \alpha\}$ be the class of stopping times whose  ``error probabilities'' are at most $\alpha$, $0 <\alpha < 1$. 
Suppose that $\Exp_i |Z_1|^2 < \infty$ and that $Z_1$ is $\Pro_i$-non-arithmetic. Then
\begin{equation}  \label{LB}
\inf_{T \in \cca} \;  \max_{i=1, \ldots, K} \;  I_{i} \,  \Exp_{i}[T]   \geq |\log \alpha| + \log \Bigl( \sum_{i=1}^{K}  \delta_{i} \, e^{\varkappa_{i}}  \Bigr) +o(1) \quad \text{as}~~ \alpha \to 0,
\end{equation}
and this asymptotic lower bound is attained by the mixture rule $T_A=T_A(p^0)$ defined in \eqref{rule2} whose mixing distribution is
\begin{equation} \label{prior}
p_{i}^0=\frac{e^{\varkappa_{i}}}{\sum_{i=1}^{K} e^{\varkappa_{i}}} \; ,  \quad \quad i=1, \ldots, K
\end{equation}
and whose error probability is exactly equal to $\alpha$, i.e., the threshold $A=A_\alpha$ is selected in such a way that $\Pro_0 (T_A(p^0)<\infty)=\alpha$.
\end{theorem}

\begin{proof}
Let $\{p_{i}\}$ be an arbitrary mixing distribution, $\pi=1/2$, $Q>1/e$ and choose $c< 1/ 2Q$ so that $\Pro_{0}(T_{A_{Qc}}<\infty)=\alpha$ (recall the definition of $A_{Qc}$ in (\ref{qc})). Then from (\ref{err}) in the appendix it follows that $\alpha \leq 2 \, Q \, c$ and  from the definition of $\cR_{c}$  we obtain the following inequality:
\begin{equation} \label{start}
\frac{\alpha}{2} + \frac{c}{2}  \; \inf_{T \in \cca} \; \max_{i=1, \ldots, K} \;  I_{i}  \, \Exp_{i}[T] \geq  \inf_{T \in \cca} \, \cR_{c}(T).
\end{equation}
By  Lemma \ref{leo5}, there exists  $c^{*}<1/Q$  such that for every $c < c^{*}/2$ (and consequently for every $\alpha < Q c^{*}$):
\begin{equation} \label{bayes1}
\inf_{T \in \cca}  \cR_{c}(T) =  \cR_{c}(T_{A_{Qc}})= \frac{\alpha}{2} +   \frac{c}{2} \,  \sum_{i=1}^{K} p_{i} \; I_{i} \, \Exp_{i}[T_{A_{Qc}}].
\end{equation}
Consequently, from (\ref{start}) and (\ref{bayes1}) it follows that
\begin{equation}
\inf_{T \in \cca} \; \max_{i=1,\ldots,K} \;  I_{i}  \, \Exp_{i}[T] \geq \sum_{i=1}^{K} p_{i} \; I_{i} \, \Exp_{i}[T_{A_{Qc}}].
\end{equation}
It remains to show that if $\{p_{i}\}$ is chosen according to (\ref{prior}), then
\begin{equation} \label{deso}
\sum_{i=1}^{K} p_{i} \; I_{i} \, \Exp_{i}[T_{A_{Qc}}]  =   |\log \alpha| + \log \Bigl( \sum_{i=1}^{K}  \delta_{i} \, e^{\varkappa_{i}} \Bigr) +o(1) \quad \text{as}~~ \alpha\to0.
\end{equation}
Substituting the mixing distribution (\ref{prior}) in (\ref{perfect2}), we obtain that, as $\alpha \to 0$,
\begin{equation} \label{des}
I_{i} \, \Exp_{i}[T_{A_{Qc}}]  =   |\log \alpha| + \log \Bigl( \sum_{i=1}^{K}  \delta_{i} \, e^{\varkappa_{i}} \Bigr) +o(1), \quad i=1,\ldots, K,
\end{equation}
which implies (\ref{LB}). Since by construction $\Pro_{0}(T_{A_{Qc}}<\infty)=\alpha$, it also follows from \eqref{des} that
\[
\max_{1 \le i \le K} I_{i} \, \Exp_{i}[T_{A}]  = |\log \alpha| + \log \Bigl( \sum_{i=1}^{K}  \delta_{i} \, e^{\varkappa_{i}} \Bigr) +o(1) \quad \text{as}~~ \alpha\to0
\]
whenever $A=A_\alpha$ is chosen so that $\Pro_0(T_A < \infty)=\alpha$. The proof is complete.
\end{proof}

Therefore, Theorem~\ref{theo2} implies that if the threshold $A=A_\alpha$ is selected so that $\Pro_0(T_A<\infty)=\alpha$ and the mixing distribution $p=p^0$ is given by \eqref{prior}, then the test $T_A(p^0)$ is third-order asymptotically minimax, i.e., as $\alpha\to 0$,
\[
\inf_{T \in \cca} \;  \max_{1 \le i \le K} \;  (I_{i} \,  \Exp_{i}[T])   =  \max_{1 \le i \le K} \;  (I_{i} \,  \Exp_{i}[T_A(p^0)]) + o(1)
\]
and
\[
 \max_{1 \le i \le K} \;  (I_{i} \,  \Exp_{i}[T_A(p^0)]) = |\log \alpha| + \log \Bigl( \sum_{i=1}^{K}  \delta_{i} \, e^{\varkappa_{i}}  \Bigr) +o(1) .
\]

\subsection{Asymptotic Minimax Performance of Mixture Rules} \label{ss:compa}

The minimax performance loss of an arbitrary mixture rule $T_{A}=T_A(p)$ with mixing prior $p=\{p_i\}$ and  error probability $\Pro_{0}(T_{A} < \infty)=\alpha$ can be naturally defined as follows:
\begin{equation} \label{loss0}
\cL_\alpha(T_A(p)) = \max_{1 \leq i \leq  K} \, (I_{i} \; \Exp_{i}[T_{A}]) - \inf_{T \in \cca} \max_{1 \leq i \leq  K} \, (I_{i} \; \Exp_{i}[T]). 
\end{equation}
Corollary \ref{cor1} implies that if $T_{A}(p)$ gives  positive weights to all of its components, i.e., $p_{i}>0$ for every $1 \leq i \leq K$, 
then:
\begin{align*}
I_{i} \; \Exp_{i}[T_{A}] &= |\log \alpha| + \log \Bigl(\sum_{j=1}^{K}  p_{j} \, \delta_{j} \Bigr) + \varkappa_{i} - \log p_{i}+o(1) , \quad 1 \leq i \leq K,
 \end{align*} 
and consequently,
\begin{equation} \label{form}
\max_{1 \leq i \leq  K} \, (I_{i} \; \Exp_{i}[T_{A}]) 
= |\log \alpha| + \log \Bigl[ \Bigl( \sum_{j=1}^{K}  p_{j} \, \delta_{j} \Bigr) \, \Bigl( \max_{1\leq i \leq K} (e^{\varkappa_{i}}/p_{i}) \Bigr) \Bigr] +o(1).
\end{equation}
Therefore, based on  (\ref{LB}) and (\ref{form}), for relatively small $\alpha$ we can approximate the performance loss (\ref{loss0}) of an arbitrary mixture rule $T_{A}(p)$ with mixing distribution $p=\{p_{i}\}$ as follows:
\begin{equation} \label{loss}
\begin{aligned}
\cL(p) &= \log \Bigl[ \Bigl( \sum_{j=1}^{K}  p_{j} \, \delta_{j} \Bigr) \, \Bigl( \max_{1\leq i \leq K} (e^{\varkappa_{i}}/p_{i}) \Bigr) \Bigr] -  
\log \Bigl[ \sum_{j=1}^{K}  e^{\varkappa_{j}} \, \delta_{j}  \Bigr]  \\
&= \log  \frac{ \Bigl( \sum_{j=1}^{K}  p_{j} \, \delta_{j} \Bigr) \, \Bigl( \max_{1\leq i \leq K} (e^{\varkappa_{i}}/p_{i}) \Bigr) }
 { \sum_{j=1}^{K} e^{\varkappa_{j}} \, \delta_{j}},
 \end{aligned}
\end{equation}
where $\cL(p) = \lim_{\alpha\to 0} \cL_\alpha(T_A(p))$ is the limiting (asymptotic) loss. 

Clearly, $\cL(p)>\cL(p^{0})=0$ for any $p=\{p_{i}\}$, where $p^{0}=\{p^{0}_{i}\}$ is the ``optimal'' mixing distribution defined in (\ref{prior}). 
Along with the uniform mixing distribution $p^{u}=\{p^{u}_{i}\}$, $p^{u}_{i}=1/K$ for every $1 \leq i \leq K$, which would be perhaps the first choice  for practical implementation, consider the following  mixing distributions:
\begin{equation} \label{more}
p^{KL}_{i}=\frac{I_{i}}{\sum_{j=1}^{K}I_{j}} \; ,  \quad p^{1/\delta}_{i}=\frac{1/\delta_{i}}{\sum_{j=1}^{K} (1/\delta_{j})} \; , \quad p^{e^{\kappa}/\delta}_{i}=\frac{e^{\kappa_{i}}/\delta_{i}}{\sum_{j=1}^{K} (e^{\kappa_{j}}/\delta_{j})} ,  \quad 1 \leq i  \leq K,
\end{equation}
which resemble  $p^{0}$ in that they all give more weight to those members of $\cP$ that are further from $\Pro_{0}$.
Notice also that in the completely symmetric case that the $\Pro_{i}$-distribution of $\Lambda_{1}^{i}$ does not depend on $i$, these mixing distributions reduce to uniform mixing $p^{u}$.  Using (\ref{loss}), we obtain
\begin{align*}
\cL(p^{KL}) &= \log  \frac{ \Bigl(\sum_{j=1}^{K} \delta_{j} I_{j} \Bigr) \; \Bigr(\max_{1\leq i \leq K} (e^{\varkappa_{i}}/{I_{i}}) \Bigl)} 
                           {\sum_{j=1}^{K} \delta_{j} \,  e^{\varkappa_{j}}} \, , \quad 
\cL(p^{1 / \delta}) = \log \frac{ K \, \Bigl(\max_{1 \leq i \leq K} (\delta_{i} \, e^{\varkappa_{i}}) \Bigr)} 
                                   {\sum_{j=1}^{K} \delta_{j} \,  e^{\varkappa_{j}}}  \\
\cL(p^{u}) &= \log  \frac{ \Bigl(\sum_{j=1}^{K} \delta_{j} \Bigr) \, \Bigl( \max_{1 \leq i \leq K} e^{\varkappa_{i}}\Bigr)}
                          {\sum_{j=1}^{K} \delta_{j} \,  e^{\varkappa_{j}}}    \, , \quad
\cL(p^{e^{\kappa}/\delta}) = \log   \frac{ \Bigl(\sum_{j=1}^{K} e^{\kappa_{j}} \Bigr) \, \Bigl(\max_{1 \leq i \leq K} \delta_{i}\Bigr)}
                                           { \sum_{j=1}^{K} \delta_{j} \,  e^{\varkappa_{j}}}   
 \end{align*}

\subsection{An Inefficient Minimax Mixture Rule}  \label{ss:inef}

We close this section by explaining why we chose to work with a ``modified'' minimax criterion, instead of $\inf_{T \in \cca} \max_{i} \, \Exp_{i}[T]$, 
which at first glance would be a more natural choice. The reason is that if we wanted to design a mixture rule $T_{A}$ that would optimize the latter criterion 
(at least asymptotically), $T_{A}$  should be an equalizer at least up to a first order, i.e. $\Exp_{i}[T_{A}]/\Exp_{j}[T_{A}]$ should be approaching 1 as $A \rightarrow \infty$  for any $1\leq i \neq j\leq K$.  However, assuming that  $I_{i}>I_{ii^*}$  and that Condition \ref{cond} holds for every $i=1, \ldots, K$, Theorem~\ref{theo1}  implies that
\begin{equation} \label{fo}
(I_{i}-I_{ii^*}) \, \Exp_{i}[T_{A}] = |\log \alpha| \, (1+o(1)), \quad i=1,\ldots,K ,
\end{equation}
where $\alpha=\Pro_{0}(T_{A}<\infty)$. Thus, a necessary condition for a mixture rule to attain $\inf_{T \in \cca} \max_{i} \, \Exp_{i}[T]$ asymptotically
is  that 
\begin{equation} \label{for}
I_{i}-I_{ii^*} = I_{j}-I_{jj^*} , \quad 1 \leq   i \neq j\leq K.
\end{equation}
But this condition is not satisfied in general by a non-trivial mixture stopping rule that gives a positive weight to all of its components. Indeed, if $p_{i}>0$ for every $i$, (\ref{for}) holds only in the completely symmetric case that $I_{1}=\ldots=I_{K}$. In general, this condition is satisfied by any mixture rule for which $$p_{i}>0 \Leftrightarrow I_{i} = \min_{j\neq i}[I_{j}-I_{jj^*}].$$ 

However, such a minimax mixture rule can be very inefficient --- it is not even uniformly first-order asymptotically optimal unless we are dealing with the symmetric case. Consider, for example, the slippage problem with $K$ populations and suppose that only one population can be out of control and that $I_1 \ll I_2=\cdots = I_K$. Then, if we wanted to attain $\inf_{T \in \cca} \max_{i} \, \Exp_{i}[T]$, even asymptotically, we should use the one-sided SPRT $T_{A}^{1}$, which is optimal under $\Pro_{1}$, but ignores all other states of the alternative hypothesis. This is clearly not a meaningful answer and shows that the seemingly natural minimax criterion $\inf_{T \in \cca} \max_{i} \, \Exp_{i}[T]$ is not appropriate.

\section{CONTINUOUS MIXTURE RULES FOR AN EXPONENTIAL FAMILY}\label{s:Expfamily}

\subsection{Notation and Assumptions}

In this section we assume that $\cP=\{\Pro_{\theta}\}_{\theta \in \bar{\Theta}}$, where $\bar{\Theta} \subset \Theta$ 
is a finite interval bounded away from 0 and that the $\Pro_{\theta}$-distribution of $X_{1}$, $F_{\theta}$, is defined by (\ref{koop}). 
Recall the definition of the likelihood ratio $\Lambda_{n}^{\theta}$ in (\ref{LR}) and write
$$
S_{n}^{\theta}= \log \Lambda_{n}^{\theta} = \theta \,  \sum_{k=1}^{n}   X_{k}- n \, \psi_{\theta}.
$$
Observe that $\Exp_{\theta}[S_{1}^{\theta}]= \Exp_{\theta}[\theta X_{1} -\psi_{\theta}]= \theta \psi'_{\theta}- \psi_{\theta}=I_{\theta}$, where
$I_{\theta}$ is the Kullback--Leibler divergence of $F_{\theta}$ and $F_{0}$. For every $\theta \in \Theta$, we define the corresponding  one-sided SPRT and overshoot
$$
T_{A}^{\theta} = \inf \Bigl\{ n \geq 1: S^{\theta}_{n}    \geq  \log A \Bigr\} ,  \quad  \eta_{A}^{\theta}=  S^{\theta}_{T^{\theta}_{A}} -\log A \quad \text{on}~~\{T^{\theta}_{A}<\infty\}.
$$
For every $\theta , \tilde{\theta} \in \Theta$ such that $\Exp_{\tilde{\theta}}[\theta  X_{1}- \psi_{\theta}]= \theta \,  \psi'_{\tilde{\theta}} - \psi_{\theta}>0$, we set
\begin{equation}
\delta_{\tilde{\theta}|\theta}= \int_{0}^{\infty} e^{-x} \, \cH_{\tilde{\theta}|\theta}(dx) , \quad \varkappa_{\tilde{\theta}|\theta}= \int_{0}^{\infty} x \,  \cH_{\tilde{\theta}|\theta}(dx) ,
\end{equation}
where $\cH_{\tilde{\theta}|\theta}$ is the asymptotic distribution of $\eta_{A}^{\theta}$ under $\Pro_{\tilde{\theta}}$, i.e.,
$\cH_{\tilde{\theta}|\theta}(x)= \lim_{A \rightarrow \infty} \Pro_{\tilde{\theta}} (\eta_{A}^{\theta} \leq x)$. For brevity's sake, we write 
$\cH_{\theta} = \cH_{\theta|\theta}$, $\varkappa_{\theta}= \varkappa_{\theta|\theta}$, and $\delta_{\theta}= \delta_{\theta|\theta}$.

From (\ref{per1}) it follows that if $\alpha=\Pro_{0}(T_{A}^{\theta} <\infty)$, then the optimal asymptotic performance under $\Pro_{\theta}$ is
\begin{equation} \label{per111}
I_{\theta}  \inf_{T\in\cca} \Exp_{\theta}[T] = I_{\theta} \, \Exp_{\theta}[T_{A}^{\theta}]= |\log \alpha| +  \log (\delta_{\theta} \, e^{\varkappa_{\theta}}) +o(1) \quad \text{as}~~\alpha\to 0.
\end{equation}

Recall that in the continuous parameter case the mixture test $T_A$ is defined by \eqref{rule} with the average likelihood ratio process $\Lambda_n$ given by \eqref{cont}. Below we assume that mixing distribution $G(\theta)$ has continuous density $g(\theta)$ with respect to the  Lebesgue measure, in which case
\begin{equation*} 
\Lambda_{n} = \int_{\Theta}   \exp \{S_n^\theta\}  \, g(\theta) \, d\theta \; ,  \quad n \in \mathbb{N}.
\end{equation*}

\subsection{Asymptotic Performance of Continuous Mixture Rules}\label{ss:TOasycontperf}

The following lemma provides a higher-order asymptotic approximation for the expected sample size $\Exp_\theta [T_A]$ for large threshold values. 

\begin{lemma} \label{leo6}
If $g$ is a positive and continuous mixing density on $\bar{\Theta}$ and $\Pro_0(T_{A}<\infty)=\alpha$, then  for every $\theta \in \bar{\Theta}$
\begin{equation} \label{p}
\begin{aligned}
I_{\theta} \; \Exp_{\theta}[T_{A}] & = |\log \alpha| +   \log \sqrt{|\log \alpha|} - \frac{1+ \log (2 \pi)}{2}
\\
& \quad + \log \Bigl( \frac{e^{\varkappa_{\theta}}  \, \sqrt{\psi''_{\theta}/ I_{\theta}}}{g(\theta)} \, \int_{\bar{\Theta}} \delta_{\theta} \, g(\theta) \, d \theta \Bigr)  +o(1) \quad \text{as}~~ \alpha \to 0.
\end{aligned}
\end{equation}
\end{lemma}

\begin{proof}
From \citet{PollakSiegmund-AS75} and \citet{Woodroofe-book82}, p.\ 68, it follows that for every $\theta \in \Theta\backslash \{0\}$
\begin{equation} \label{polla}
I_{\theta} \; \Exp_{\theta}[T_{A}]= \log A +   \log \sqrt{\log A} - \frac{1+ \log (2 \pi)}{2}
+ \log \Bigl( e^{\varkappa_{\theta}} \frac{\sqrt{\psi''_{\theta}/I_{\theta}}}{g(\theta)} \Bigr) +o(1) \quad \text{as}~~A\to\infty.
\end{equation}

Moreover, from Corollary 1  in \citet{Woodroofe-book82}, p.\ 67  (see also \citet{Pollak-AS86}) it follows  that
$$  
A \, \Pro_0(T_{A}<\infty) \rightarrow \int_{\bar{\Theta}} \delta_{\theta} \, g(\theta) \, d \theta ,
$$
and consequently,  
\begin{equation} \label{wood}
\log A = |\log \alpha| +  \log  \Bigl( \int_{\bar{\Theta}} \delta_{\theta} \, g(\theta) \, d \theta \Bigr) +o(1).
\end{equation}
We can now complete the proof by substituting (\ref{wood}) into (\ref{polla}).
\end{proof}

Asymptotic approximations \eqref{per111} and \eqref{p} imply that any continuous mixture rule with positive and continuous density on $\bar{\Theta}$ minimizes the expected sample size to first-order for every $\theta \in \bar{\Theta}$, i.e.,
\[
\Exp_{\theta}[T_{A}]= \inf_{T\in  \cca} \Exp_\theta [T] \, (1+o(1)) \quad \text{as}~~\alpha \to 0~~\text{for all}~~\theta \in \bar{\Theta}.
\] 
However, such a continuous mixture rule is  not second-order asymptotically optimal  for any
$\theta \in \bar{\Theta}$. More specifically, the following asymptotic equality holds 
$$  
\Exp_{\theta}[T_{A}]- \inf_{T\in  \cca} \Exp_\theta [T] = O\brc{\log (\sqrt{|\log \alpha|})} \quad \text{for all}~~\theta \in \bar{\Theta}.
$$
In other words, the distance between $\Exp_{\theta}[T_{A}]$ and the optimal asymptotic performance  (\ref{per111}) under $\Pro_{\theta}$
does not remain bounded as $\alpha\to 0$ for  any $\theta \in \bar{\Theta}$.

\subsection{A Nearly Minimax Continuous Mixture Rule}\label{ss:TOminimaxexp}

In the following theorem we show that a particular \textit{continuous} mixture rule is third-order asymptotically minimax in the sense of minimizing the maximal Kullback--Leibler information $\sup_\theta I_\theta \Exp_\theta[T]$ in the class $\cca$ as $\alpha\to0$. 

\begin{theorem} \label{theo3}
If the limiting average overshoot $\varkappa_{\theta}$ is a continuous function on $\bar{\Theta}$, then  
\begin{equation}  \label{pol}
\begin{aligned}
\inf_{T \in \cca} \;  \sup_{\theta \in \bar{\Theta}} \quad I_{\theta} \; \Exp_{\theta}[T]
&\geq  |\log \alpha| +  \log \sqrt{|\log \alpha|} -  \frac{1+ \log (2 \pi)}{2} 
\\
& \quad + \log \Bigl( \int_{\bar{\Theta}} \delta_{\theta} \,  e^{\varkappa_{\theta}} \, \sqrt{\psi''_{\theta} / I_{\theta}} \; d \theta \Bigr) + o(1) \quad \text{as}~~ \alpha\to 0,
\end{aligned}
\end{equation}
and this asymptotic lower bound is attained by the continuous mixture rule $T_A(g^0)$ whose mixing density is
\begin{equation} \label{prior2}
g^0(\theta)  =  \frac{e^{\varkappa_{\theta}}  \sqrt{\psi''_{\theta}/ I_{\theta}}}
{\int_{\bar{\Theta}}   e^{\varkappa_{\theta}} \; \sqrt{\psi''_{\theta} / I_{\theta}} \; d \theta} \, , \quad \theta  \in \bar{\Theta}
\end{equation}
and for which $\Pro_{0}(T_{A}(g^0)<\infty)=\alpha$.
\end{theorem}

\begin{proof}
Lower bound (\ref{pol}) can be established following the same steps as in the proof of Theorem \ref{theo2}. The details are omitted. 

In order to show that the mixture rule $T_A(g^0)$ with mixing density (\ref{prior2}) attains the asymptotic lower bound in (\ref{pol}), 
it suffices to substitute (\ref{prior2}) into (\ref{p}) to obtain that for every $\theta \in \bar{\Theta}$
\begin{equation} \label{mini}
\begin{aligned}
I_{\theta}  \, \Exp_{\theta}[T_{A}] &= |\log \alpha| +  \log (\sqrt{|\log \alpha|}) -  \frac{1+ \log (2 \pi)}{2}  
\\
& \quad + \log \Bigl( \int_{\bar{\Theta}} \delta_{\theta} \,  e^{\varkappa_{\theta}} \, \sqrt{\psi''_{\theta} / I_{\theta}} \; d \theta \Bigr) + o(1) \quad \text{as}~~ \alpha\to 0 .
\end{aligned}
\end{equation}
This completes the proof.
\end{proof}

\begin{remark} 
Note that for (\ref{polla}) and (\ref{mini}) to hold, the mixing density  (\ref{prior2}) must be continuous, which requires that $\varkappa_{\theta}$ must be a continuous function, since $\psi_{\theta}$ and  $I_{\theta}= \theta \psi'_{\theta}- \psi_{\theta}$ are continuous. This is true at least when the distribution of $S_1^\theta$ is continuous.
\end{remark}

Typically, the computation of the optimal mixing density (\ref{prior2}) requires discretization. An example where such a discretization is not necessary is that of an exponential distribution. More specifically, suppose that $dF_{0}(x)=e^{-x} dx$ and $dF_{\theta}(x)= e^{-(1-\theta)x} dx$ for every $0 <\theta <1 $. Then $\psi_{\theta}= - \log (1-\theta)$, $I_{\theta}= \theta/ (1-\theta) + \log (1-\theta)$ 
and the exact distribution of the overshoot $\eta_{A}^{\theta}$ is exponential with rate $(1-\theta)/\theta$ for every $A>1$. 
Therefore, $\cH_{\theta}$ is an exponential distribution with  rate $(1-\theta)/\theta$, which implies that $\varkappa_{\theta}= \theta/ (1-\theta)$ and $\delta_{\theta}= \theta$. As a result, mixing density (\ref{prior2}) is completely specified up to the normalizing constant 
\begin{equation} \label{pro}
\int_{\bar{\Theta}}   e^{\varkappa_{\theta}} \; \sqrt{\psi''_{\theta} / I_{\theta}} \; d \theta= \int_{\bar{\Theta}}  \frac{\exp\{\theta /(1-\theta)\}}{ \sqrt{(1-\theta) [\theta+ (1-\theta) \log (1-\theta)]}} d\theta ,
\end{equation}
which can be computed numerically. 

Unfortunately, $\varkappa_{\theta}$ and $\delta_{\theta}$ do not have analogous closed-form expressions in terms of $\theta$
in general. Therefore, it is typically difficult to compute  optimal mixing density $g^{0}$.  Thus, in practice it may be more convenient to choose 
mixing density $g$ from the class of probability density functions on the whole parameter space $\Theta$ that are \textit{conjugate} to $f_{\theta}$, so that the resulting mixture rule is easily computable. However, such a mixture rule will only be second-order asymptotically minimax over $\bar{\Theta}$, as it was shown by \citet{Pollak-AS78}.

In the following subsection, we consider another alternative to the nearly minimax continuous mixture rule; we approximate $\bar{\Theta}$
with a discrete set of points and we use the corresponding nearly minimax discrete mixture test.

\subsection{A Discrete Approximation}

A practical alternative to the optimal continuous mixture rule is to approximate the interval $\bar{\Theta}$  by a  genuinely discrete set, 
$\Theta_{K}= \{\theta_{1}, \ldots, \theta_{K} \}\subset \bar{\Theta}.$ In this case, the discrete mixture likelihood ratio statistic takes the form
$$
 \Lambda_{n}=\sum_{i=1}^{K} p_{i}\, e^{S_{n}^{\theta_{i}}}=  \sum_{i=1}^{K} p_{i} \exp\set{\sum_{m=1}^{n} \Bigl[ \theta_{i} \, X_{m}- \psi_{\theta_{i}} \Bigr]}, \quad n \in \mathbb{N}, 
$$ 
and, according to Theorem~\ref{theo2}, the optimal mixing distribution $\{p_{i}\}$ is given by (\ref{prior}). 
By Corollary~\ref{cor1}, such a discrete mixture rule is \textit{second-order} asymptotically optimal under $\Pro_{\theta_{i}}$ for every $i=1,\ldots,K$, that is, $\Exp_{\theta_{i}}[T_{A}] = \inf_{T \in \cca} \Exp_{\theta_{i}}[T] + O(1)$ for every $i=1,\ldots,K$. 
Moreover, it is asymptotically {\em third-order} minimax with respect to the Kullback--Leibler information, i.e., 
$$\max_{1 \leq i \leq K} (I_{\theta_{i}}\Exp_{\theta_{i}}[T_{A}]) = \inf_{T \in \cca} \max_{1 \leq i \leq K} \,  (I_{\theta_{i}} \, \Exp_{\theta_{i}}[T]) + o(1).$$
However, it is  not even first-order asymptotically optimal under $\Pro_{\theta}$ when $\theta \notin \Theta_{K}$. 
More specifically, we have the following corollary of Theorem \ref{theo1}, for which we write $I_{\theta \theta^{*}}$ for the Kullback--Leibler divergence of the distributions $F_{\theta}$ and $F_{\theta^{*}}$, that is,
\begin{equation} \label{dl}
I_{\theta \theta^{*}}= \Exp_{\theta}[ S_{1}^{\theta}- S_{1}^{\theta^{*}}] =
\Exp_{\theta}[ (\theta-\theta^{*}) \, X_{1}- (\psi_{\theta}- \psi_{\theta^{*}})]
=  (\theta-\theta^{*}) \, \psi'_{\theta}- (\psi_{\theta}- \psi_{\theta^{*}}).
\end{equation}

\begin{corollary} \label{cor2}
Suppose that $\theta \in \Theta \backslash \Theta_{K}$ and that there exists a unique $\theta^{*}= \arg \min_{\theta_{j} \in \Theta_{K}} I_{\theta \theta_{j}}$.
If $\psi_{\theta^{*}} < \theta^{*} \psi'_{\theta}$, then $\Pro_{\theta}(T_{A}<\infty)=1$. If also $\Pro_{0}(T_{A}<\infty)=\alpha$, then  
\begin{equation} \label{perfecto}
[I_{\theta}-I_{\theta \theta^{*}}] \; \Exp_{\theta}[T_{A}] =   | \log \alpha| + \log \Bigl(\sum_{i=1}^{K}  p_{i} \, \delta_{\theta_{i}} \Bigr) + \varkappa_{\theta|\theta^{*}} - \log p_{\theta^{*}}+ o(1)  \quad \text{as} \; \alpha \rightarrow 0.
\end{equation}
\end{corollary}

\begin{proof}
From Lemma \ref{leo2} it follows that $\Pro_{\theta}(T_{A}<\infty)=1$ as long as $I_{\theta}> I_{\theta \theta^{*}}$, or equivalently,
$$ 
\theta \psi'_{\theta}- \psi_{\theta}  >  (\theta-\theta^{*}) \, \psi'_{\theta}- (\psi_{\theta}- \psi_{\theta^{*}})
\Leftrightarrow \; \psi_{\theta^{*}} < \theta^{*} \psi'_{\theta}.
$$
Moreover, since the random variable $\theta^{*} X_{1}- \psi_{\theta^{*}}$ has non-arithmetic distribution with exponential moments under $\Pro_{\theta}$ for almost every $\theta$  (see Lemma 6.4 in \citet{Woodroofe-book82}), the conditions of Theorem \ref{theo1} are satisfied, and consequently, we obtain (\ref{perfecto}).
 \end{proof}

\section{MONTE-CARLO SIMULATIONS  }\label{s:MC} 

\renewcommand{\captionlabelfont}{\bfseries} 

In this section, we illustrate the asymptotic formulas obtained in Section~\ref{s:DMR} and check their validity  
with simulation experiments in the Gaussian example where $F_{0}(x)=\Phi(x)$ and  $F_{i}(x)= \Phi(x-i)$ for $i=1,2,3$ ($\Phi(x) = (2\pi)^{-1/2}\int_{-\infty}^x e^{-t^2/2} dt$ is the standard normal distribution function). Thus, the observations are normally distributed with unit variance and  mean that is equal to 0 under $\Hyp_{0}$ and is either 1 or 2 or 3 under $\Hyp_{1}$ ($K=3$).  In this example, the quantities $\varkappa_{i}$ and $\delta_{i}$ can be computed with any precision using  the following expressions:
\begin{align}
\varkappa_{i} &= 1 + \frac{i^{2}}{4}- i \, \sum_{n=1}^{\infty} \Bigl[ \frac{1}{\sqrt{n}} \, \phi \Bigl( \frac{i}{2} \, \sqrt{n} \Bigr) 
                  - \frac{i}{2} \, \Phi \Bigl( -\frac{i}{2} \, \sqrt{n} \Bigr)  \Bigr],  \label{kap1} 
                  \\
\delta_{i}    &= \frac{1}{I_{i}} \exp\Bigl\{ -2 \, \sum_{n=1}^{\infty} \frac{1}{n} \, \Phi \Bigl( -\frac{i}{2} \, \sqrt{n} \Bigr) \Bigr\} .  \label{kap2}
\end{align}
(see, e.g., \citet{Woodroofe-book82}, p.\ 32). 

In Table \ref{tab1}, we compute these quantities, the optimal mixing distribution  (\ref{prior}),  as well as the mixing distributions that we defined in (\ref{more}). Using Table \ref{tab1}, we can compute the asymptotic performance loss (\ref{loss}) for each of the corresponding mixture rules:
$$
\cL(p^{KL})=0.21 \, , \quad  \cL(p^{1/\delta})=0.58 \, , \quad \cL(p^{e^{\kappa}/\delta})=0.85 \, ,  \quad \cL(p^{u})=1.21.
$$

\begin{table}[!h] 
\renewcommand{\arraystretch}{1.2}
\centering
	\caption{Mixing distributions and quantities $\varkappa_i$ and $\delta_i$} \label{tab1}	
	\vspace{3mm}
		\begin{tabular}{c||c|c|c||c|c|c|c|c} 
		$i$ & $I_{i}$  & $\varkappa_{i}$ & $\delta_{i}$ & $p^{e^{\kappa} / \delta}$ & $p^{0}$ & $p^{KL}$ & $p^{1 / \delta}$ & $p^{u}$ \\
				\hline \hline
		$1$ & $0.5$   & $0.718$ & $0.560$  & 0.25 & 0.066 & 0.071  &  0.176  & 0.33      \\
	  $2$ & $2$   &   $1.747$ & $0.320$ &  0.125 & 0.185 &  0.286 &  0.307  & 0.33   \\
	  $3$ & $4.5$ & $3.146$ &  $0.190$  & 0.85 &  0.749 &  0.643 &  0.517  & 0.33    \\
		\end{tabular}
	\end{table}

In Remark \ref{remar} we discussed that if we set $A$ as
\begin{equation} \label{how}
A = \frac{\sum_{i=1}^{K} p_{i}\,  \delta_{i}}{\alpha},
\end{equation}
where $p=\{p_{i}\}$ is the mixing distribution that defines $T_{A}(p)$, the probability $\Pro_{0}(T_{A}(p)<\infty)$ is expected to be  approximately equal to $\alpha$ for sufficiently small values of $\alpha$. In Table \ref{tab2}, we present the actual probabilities computed using Monte Carlo simulations. An  importance sampling technique was used in these experiments, taking advantage of the representation  $\Pro_{0}(T_{A}< \infty) =\sum_{i}  p_i  \, \Exp_{i}[e^{-Z_{T_{A}}}]$ (see (\ref{com2})). This allowed us to evaluate a very low error probability with a reasonable number of Monte Carlo runs. It is seen that the formula \eqref{how} ensures extremely high accuracy of the approximation of the desired error probability for    all mixing distributions.

\begin{table}[h] 
\renewcommand{\arraystretch}{1.2}
\centering
\caption{Probability $\Pro(T_A(p)<\infty)$ for different mixing distributions: the first column represents the desired error probabilities; the other columns represent the actual error probabilities obtained by Monte Carlo simulations when  the threshold is chosen according to (\ref{how})} 
\vspace{3mm}
	\begin{tabular}{c||c|c|c|c|c} 
		$\alpha$ &  $p^{e^{\kappa} / \delta}$ & $p^{0}$ & $p^{KL}$ & $p^{1 / \delta}$ & $p^{u}$ \\
				\hline \hline
		$10^{-1}$ &   $5.9979  \, 10^{-2}$   & $6.7037 \, 10^{-2}$       &  $8.0337 \, 10^{-2}$      &  $ 8.0029 \, 10^{-2}$  &   $8.9314 \,  10^{-2} $  \\
		$10^{-2}$ &   $9.1127  \, 10^{-3}$  &  $9.4317  \, 10^{-3}$       & $9.8754 \, 10^{-3}$      &  $ 9.8885  \, 10^{-3}$  &  $1.0049 \, 10^{-2}$  \\
	  $10^{-4}$ &   $1.0104  \, 10^{-4}$  &  $1.0107 \,  10^{-4}$       & $1.0027 \, 10^{-4}$      &  $ 1.0038 \, 10^{-4}$  &   $1.0011 \, 10^{-4}$  \\
	  $10^{-6}$ &   $1.0017  \, 10^{-6}$  &  $1.0006 \, 10^{-6}$       &  $1.0009 \, 10^{-6}$      &  $ 1.0004 \,  10^{-6}$  &  $1.0008 \, 10^{-6}$   \\
	  $10^{-8}$ &   $1.0008  \, 10^{-8}$ &   $1.0033 \, 10^{-8}$       &  $1.0002 \, 10^{-8}$      &  $ 1.0017 \, 10^{-8}$   &  $1.0006 \, 10^{-8}$   \\
	\end{tabular}
	\label{tab2}
	\end{table}

Table~\ref{tab3} allows us to verify the accuracy of the asymptotic approximation  (\ref{form}) for the Kullback--Leibler information $\max_{i} (I_{i} \Exp_{i}[T_{A}(p)])$ in the  worst-case scenario for optimal mixing distribution $p=p^{0}$ and  uniform mixing distribution $p=p^{u}$.  For optimal mixing distribution $p^{0}$, the asymptotic approximation (\ref{form}) for $\max_{i} (I_{i} \Exp_{i}[T_{A}])$  is very accurate for all studied probabilities of error $\alpha \le 0.01$. However, for uniform mixing distribution, the approximation  (\ref{form})  is considerably less accurate, but improves significantly as  the error probability goes to 0. 

\begin{table}[h] 
\renewcommand{\arraystretch}{1.2}
\centering
\caption{The maximal expected Kullback--Leibler information $\max_{i} (I_{i} \Exp_{i}[T_{A}(p^0)])$ for optimal and uniform mixing distributions $p^0$ and $p^u$. The threshold $A$ is selected according to (\ref{how}).}
\vspace{3mm}
		\begin{tabular}{ccc} 
		(a) Optimal mixing distribution \vspace{3mm} & & (b) Uniform mixing distribution\\ 
		\vspace{5mm}
	\begin{tabular}{c||c|c} 
		$\alpha$ &   Monte Carlo  & Approximation \eqref{form}   \\
				\hline \hline
	  $10^{-1}$ &   4.99   &  4.31  \\	
	  $10^{-2}$ &   6.36  &   6.61  \\
	  $10^{-4}$ &   10.99 &   11.21 \\
	  $10^{-6}$ &   15.65 &   15.82 \\
	  $10^{-8}$ &   20.33 &   20.42 \\
	 \end{tabular} 
	& 
	&
	\begin{tabular}{c||c|c} 
		$\alpha$ &  Monte Carlo & Approximation  \eqref{form}   \\
				\hline \hline
		$10^{-1}$ &  5.04  &  5.52   \\		
		$10^{-2}$ &  6.88  &  7.82   \\
	  $10^{-4}$ &  11.87 &  12.42  \\
	  $10^{-6}$ &  16.59 &  17.03 \\
	  $10^{-8}$ &  21.29 &  21.63  \\
	\end{tabular} 
	\label{tab3}
\end{tabular}
	\end{table}

\section{EXTENSIONS} \label{s:Rami}

Despite the fact that one-sided tests have limited practical applications themselves, they can be  used effectively in the more realistic problems of 
testing two (or more) hypotheses and in changepoint detection problems. Indeed, multi-hypothesis sequential tests and changepoint detection procedures
are typically built based on  combinations of one-sided tests; see, e.g., \citet{Lorden-AS71,Lorden-AS77}, \citet{Tartakovskyetal-IEEEIT03}, and \citet{TartakovskySISP98}. Therefore, the results of the present paper may have certain implications for these more practical problems, some of which we now briefly discuss.

\subsection{Two-Sided Mixture Sequential Tests}

Suppose that we want to stop as soon as possible  not only under $\cP$ but also under $\Pro_{0}$ and either reject $\Hyp_0$ or accept it. Then, a sequential test is a pair $(T,d_{T})$ that  consists of an $\{\cF_{n}\}$-stopping time $T$ and an $\cF_{T}$-measurable random variable $d_{T}$ that takes values in $\{0,1\}$, depending on whether the null or the alternative hypothesis is accepted. When $\cP$ consists  of a single probability measure, say  $\cP=\{\Pro_{i}\}$,  
the optimal test is Wald's two-sided SPRT
\[
\begin{aligned} T_{A,B}^{i} & = \inf \Bigl\{n \geq 1: \Lambda_{n}^{i} \geq A \; \text{or} \; \Lambda_{n}^{i} \leq B \Bigr\} ;  \\ 
d_{T_{A,B}^{i}}&=\begin{cases}  & 1~~ \text{if} ~~ \Lambda^{i}_{T_{A,B}^{i}} \geq A   \\
& 0 ~~ \text{if}~~ \Lambda_{T_{A,B}^{i}}^{i} \leq B  \end{cases} ,
\end{aligned} \]
where $0<B<1< A$ are fixed thresholds. Indeed, as it was shown by \citet{WaldWolfowitz48}, the SPRT attains both 
$$\inf_{(T,d_{T}) \; \in \ccab^{i}} \Exp_0[T] \quad \text{and} \quad \inf_{(T,d_{T}) \; \in \ccab^{i}}  \Exp_{i}[T],$$
where  $\Pro_{0} (d_{T_{A,B}^{i}}=1)=\alpha$ , $\Pro_{1} (d_{T_{A,B}^{i}}=0)=\beta$ and $$\ccab^{i}=\Bigl\{ (T,d): \Pro_{0}(d_{T}=1) \leq \alpha \; \text{and} \; \Pro_{i} (d_{T}=0) \leq \beta \Bigr\}.$$ When the alternative hypothesis consists of a discrete set of probability measures, $\cP=\{\Pro_{1}, \ldots, \Pro_{K}\}$,
a natural generalization of the SPRT  is the two-sided mixture rule 
\[
\begin{aligned}
T_{A,B}=\min \set{T_0(B), T_1(A)} \; , \; d_{T_{A,B}}= \ind{T_1(A)< T_0(B)},
\end{aligned}
\]
where 
\[ 
T_0(B)  =  \inf \Bigl\{n \geq 1: \sum_{i=1}^{K} q_{i} \Lambda_{n}^{i} \leq B \Big\}, \quad  T_1(A)=\inf \Bigl\{n \geq 1: \sum_{i=1}^{K} p_{i} \Lambda_{n}^{i} \geq A \Big\}
\] 
and $\{q_{i}\}$, $\{p_{i}\}$ are mixing distributions. We conjecture that if $\{p_{i}\}$  is chosen according to (\ref{prior}), 
then $(T_{A,B},d_{T_{A,B}})$ is almost minimax, in the sense that it attains 
$$
\inf_{(T,d_{T}) \; \in \ccab}  \; \max_{i=1,\ldots, K} \, (I_{i} \Exp_{i}[T])
$$
up to an $o(1)$ term  as $\alpha |\log \beta|+ \beta |\log \alpha| \rightarrow 0$, where
$\Pro_{0} (d_{T_{A,B}}=1)=\alpha$ and  $\Pro_{1} (d_{T_{A,B}}=0)=\beta$ and 
$$ 
\ccab=\Bigl\{ (T,d_{T}): \Pro_{0}(d_{T}=1) \leq \alpha \; \text{and} \; \max_{i=1,\ldots,K} \;  \Pro_{i} (d_{T}=0) \leq \beta \Bigr\}.$$
However, this statement does not follow directly from our results in this paper. Moreover, it is not clear whether  
$\inf_{(T,d_{T}) \; \in \ccab}  \Exp_{0}[T]$ is attained up to an $o(1)$ term for some particular choice of $\{q_{i}\}$. This open problem will be addressed in the  future.

\subsection{Sequential Changepoint Detection}

Suppose that a change occurs at an unknown time $\nu$ so that the pre-change distribution of the  sequence $\{X_{n}\}$ is  $F_{0}$ 
and the post-change distribution belongs to the set $\{F_{1},\ldots, F_{K}\}$. We denote by $\Pro_{i}^{\nu}$  the probability measure under which the change occurs at time $\nu$ and the post-change distribution is $F_{i}$. If $\nu=\infty$ (there is never a change), then $X_n \sim F_{0}$ for every $n \in \mathbb{N}$, i.e., $\Pro_i^\infty\equiv \Pro_0$. If $\nu=1$ (the change occurs at the very beginning), then $X_n \sim F_i$ for all $n \in \mathbb{N}$, i.e., $\Pro_i^1=\Pro_i$.  The goal is to detect the change as soon as possible after it occurs, avoiding false alarms. Thus, a detection rule is a stopping time $T$, and one attempts to find such $T$ that $(T-\nu)^{+}$ takes small values under every $\Pro_{i}^{\nu}$, but large values under $\Pro_{0}$. 

\citet{Lorden-AS71} showed that there is a close link between change detection rules and one-sided sequential tests. 
Based on this connection, he proved that applying repeatedly the one-sided SPRT, $T_{A}^{i}$, leads to a detection rule
(the so-called CUSUM procedure) that is asymptotically optimal in the sense that it attains to first order
\begin{equation} \label{lor}
\inf_{T: \Exp_{0}[T] \geq A} \cJ_{i}[T],
\end{equation}
where $\cJ_{i}[T]$ is a minimax performance measure that quantifies the delay of the detection rule $T$ when the post-change distribution is $F_{i}$.
Using Lorden's method, it can be easily established that applying repeatedly a mixture-based sequential test $T_{A}$ with $p_{i}>0$ for all $i=1,\ldots, K$ leads to a detection procedure that attains to first order (\ref{lor}) for every $i=1,\ldots, K$. However, the optimal choice of the mixing distribution 
remains an open problem that we plan to consider in the future.

\section{CONCLUSIONS AND FINAL REMARKS} \label{s:Conclu}
The main focus of this paper is on discrete, mixture-based stopping rules for testing a simple null hypothesis against a composite alternative hypothesis. 
These rules arise naturally in important practical problems, such as the multi-sample slippage problem, where the statistician has to decide whether
one of the populations has ``slipped to the right of the rest'', without specifying which one. Discrete mixture rules are also useful when the 
alternative hypothesis  is continuous, since they have certain important advantages over their continuous counterparts. More specifically, they asymptotically minimize  the expected sample size within a constant (not only to first-order)  at all parameter values used for their design 
(but they are asymptotically suboptimal outside of these points). However, the most important advantage of discrete mixtures is that they are easily implementable, which is not usually the case with continuous mixture rules. 

The main contribution of this paper consists in finding an optimal mixing distribution both for discrete and continuous mixture rules. That is, for both cases, we find mixing distributions so that the resulting sequential tests are nearly minimax, in the sense that they minimize the maximal Kullback--Leibler information within a negligible term $o(1)$. We believe that the  methods of the present paper can be effectively used in the more practical problems of sequential testing two or more composite hypotheses and constructing nearly optimal mixture-based change-point detection procedures.

\section*{APPENDIX: PROOF OF LEMMA 2.5}

\renewcommand{\theequation}{A.\arabic{equation}}
\setcounter{equation}{0}

We need to find a $c^{*}$  such that $\cR_c(T) \geq \cR_{c}(T_{A_{Qc}})$ for every stopping time $T$ and for every $c$ smaller than
$\pi c^{*}$, or equivalently, for every $c$ that satisfies the inequality $Qc< \pi Qc^{*}$. Since $A_{Qc}$ is defined so that $Qc  < \pi$, it is clear that $c^{*}$ must be chosen so that $Q c^{*} <1$.

Recalling that $\pi=\Prop^\pi(\Hyp_0)$ is the prior probability of the null hypothesis $\Hyp_0$ as well as the definitions of the probability measure $\Prop^{\pi}$ and the posterior process $\{\Pi_{n}\}_{n\ge 1}$, for any stopping time $T$, we have
\begin{equation} \label{error}
\pi \,  \Pro_{0}(T < \infty) =  \sum_{n=1}^{\infty} \Prop^{\pi}( T=n,  \theta=0)
= \sum_{n=1}^{\infty} \Expop^{\pi}[ \ind{T=n} \,  \Pi_{n}] =  \Expop^{\pi}[ \Pi_{T}  \ind{T < \infty}]
\end{equation}
and
$$
c \, (1-\pi) \,  \sum_{i=1}^{K} p_{i} \, (I_{i} \, \Exp_{i}[T])
\geq c \; (\min_{1\leq i \leq K} I_{i}) \;  (1-\pi) \,  \sum_{i=1}^{K} p_{i} \, \Exp_{i}[T]
= c \; (\min_{1 \leq i \leq K} I_{i}) \;  \Expop^{\pi}[T] .
$$
Therefore,  
$$
\cR_{c}(T) \geq \Expop^{\pi}[ \Pi_{T}  \ind{T < \infty} + c \; ( \min_{1 \leq i \leq K} I_{i} ) \, T].
$$
From this inequality it is clear that without any loss of generality we can restrict ourselves  to $\Prop^{\pi}$-a.s.  finite stopping times. Since the process $\{\Pi_{n}\}_{n\ge 0}$ is a bounded martingale  with $\Pi_{0}=\pi$, we conclude that
$\cR_{c}(T) \geq \Expop^{\pi}[ \Pi_{T} ] =\pi$   for every  $\Prop^{\pi}$-a.s.\  finite stopping time $T$.
Hence, it suffices to find $c^{*}$ with $Q c^{*}<1$ such that for every $c \leq \pi \, c^{*}$
$$ 
\pi \geq \cR_{c}(T_{A_{Qc}}) = \pi \, \Pro_{0}(T_{A_{Qc}}<\infty) +  c \, (1-\pi) \, \sum_{i=1}^{K} p_{i} \, (I_{i} \, \Exp_{i}[T_{A_{Qc}}]).
$$
From (\ref{repre})  and (\ref{error}) it follows that
\begin{equation} \label{err}
\pi \, \Pro_{0}(T_{A_{Qc}}<\infty)= \Expop^{\pi}[\Pi_{T_{A_{Qc}}}] \leq Qc.
\end{equation}
Therefore, we must find $c^{*}$ with  $Q c^{*}  \leq 1$ such that for every $c \leq \pi c^{*}$
\begin{equation} \label{err2}
 Q \, c +  c \, (1-\pi) \, \sum_{i=1}^{K} p_{i} \, (I_{i} \, \Exp_{i}[T_{A_{Qc}}]) \leq \pi
\Longleftrightarrow   (1-\pi) \, \sum_{i=1}^{K} p_{i} \, (I_{i} \, \Exp_{i}[T_{A_{Qc}}]) \leq \frac{\pi}{c} -Q .
\end{equation}
However,  from (\ref{perf10}) it follows that there exists a constant $C>0$, which does not depend on $i$ and $A$,  such that
$I_{i} \, \Exp_{i}[T_{A}] \leq \log A + C$ for any mixture rule $T_{A}$. Therefore,
\begin{align*}
(1-\pi) \, \sum_{i=1}^{K} p_{i} \, (I_{i} \, \Exp_{i}[T_{A_{Qc}}])
&\leq  (1-\pi) [ \log A_{Qc} + C ] \nonumber \\
&=  (1-\pi) \Bigl[ \log \Bigl( \frac{1-Qc}{Qc}  \frac{\pi}{1-\pi}  \Bigr)+   C  \Bigr] \\
&\leq 
(1-\pi) \Bigl[ \log \Bigl(\frac{\pi}{Qc}\Bigr)  + \log \Bigl(\frac{1}{1-\pi}\Bigr) +   C  \Bigr] \\
&\leq    \log \Bigl( \frac{\pi}{Qc} \Bigr) +  (1-\pi)  \log \Bigl( \frac{1}{1-\pi} \Bigr) +   C  \\
&=   \frac{\pi}{Qc}   \Bigl[ \frac{Qc}{\pi}  \log \Bigl( \frac{\pi}{Qc} \Bigl) \Bigr] +  \Bigl[(1-\pi)  \log \Bigl( \frac{1}{1-\pi} \Bigr) \Bigr]+C .
\end{align*}
Since also $Qc \leq \pi$, from the inequality $\sup_{0 < x <1} \Bigl( x |\log x| \Bigr) \leq e^{-1}$ we have
\begin{equation} \label{err3}
(1-\pi) \, \sum_{i=1}^{K} p_{i} \, (I_{i} \, \Exp_{i}[T_{A_{Qc}}])
\leq   \frac{\pi}{Qc}  \frac{1}{e} +  \frac{1}{e} +   C
=  \frac{\pi}{c} -  \frac{\pi}{c}  \, \frac{Qe-1}{Q e} \,  +  \frac{1}{e} +  C .
\end{equation}
Hence, from (\ref{err2}) and (\ref{err3}) it follows that it suffices to find $c^{*}$ with $Q c^{*}<1$  such that for $c \leq \pi  \, c^{*}$
\begin{align*}
\frac{\pi}{c} -  \frac{\pi}{c}  \, \frac{Qe-1}{Q e} \,  + e^{-1} +   C  \leq \frac{\pi}{c}- Q
&\Longleftrightarrow    \frac{\pi}{c}  \, \frac{Qe-1}{Q e} \geq   e^{-1} +  Q+ C  \\
&\Longleftrightarrow    \frac{c}{\pi} \leq  \frac{Qe-1}{Q e}  \frac{1} {e^{-1} +  Q+ C} .
\end{align*}
Thus, it suffices to set
$$
c^{*}=  \frac{Qe-1}{Q e}  \frac{1} {e^{-1} +   Q + C } ,
$$
and this is a valid choice since
$$ 
Q c^{*} \leq \, \frac{Qe-1}{Qe} \, \frac{Q}{e^{-1}+   Q + C}  < 1 .
$$
The proof is complete.

\section*{ACKNOWLEDGEMENTS}
We would like to thank Nitis Mukhopadhyay for inviting us to submit a paper to this special issue. We are also grateful to Moshe Pollak for useful discussions.

This work was supported by the U.S.\ Army Research Office under MURI grant  W911NF-06-1-0044, by the U.S.\ Air Force Office of Scientific Research under MURI grant FA9550-10-1-0569, by the U.S.\ Defense Threat Reduction Agency under grant HDTRA1-10-1-0086, and by the U.S.\ National Science Foundation under grants CCF-0830419 and EFRI-1025043 at the University of Southern California, Department of Mathematics.


\begin{thebibliography}{99}

\setlength{\parskip}{-1.2ex plus0ex minus1.1ex}


\bibitem[Anscombe(1952)]{Ans1952} Anscombe, F.J. (1952). Large-Sample Theory of Sequential Estimation, \textit{Proc. Cambridge Philos. Soc.} 48: 600-607.


\bibitem[Chow et al.(1971)]{ChowRobbinsSiegmund-book71} \text{Chow, Y.S. , Robbins, H. and Siegmund, D.} (1971). \textit{Great Expectations: The Theory of Optimal Stopping}, Boston:Houghton Mifflin.

\bibitem[Darling and Robbins(1968)]{DarlingRobbins-NAS68} \text{ Darling, D. and  Robbins, H.}(1968). Some Further Remarks on Inequalities for Sample Sums, \textit{Proceedings of the National Academy of Science of the U.S.A.} 60: 1175--1182.


\bibitem[Dragalin and Novikov(1999)]{DragNov} Dragalin, V. and Novikov, A. (1999). Adaptive Sequential Tests for Composite Hypotheses, \textit{Surveys in Applied and Industrial Mathematics, TVP press}, 6: 387-–398.

\bibitem[Dragalin et al.(1999)]{DTVIEEE99}  Dragalin, V. P., Tartakovsky, A. G., and Veeravalli, V. V. (1999). Multihypothesis Sequential Probability Ratio Tests - Part I: Asymptotic Optimality,  \textit{IEEE Transactions on Information Theory} 45: 2448–-2461.

\bibitem[Dragalin et al.(2000)]{dragalin-it00} Dragalin, V. P., Tartakovsky, A. G., and Veeravalli, V. V. (2000). Multihypothesis Sequential Probability Ratio Tests - Part II: Accurate Asymptotic Expansions for the Expected Sample Size, \textit{IEEE Transactions on
Information Theory} 46: 1366–-1343.

\bibitem[Lai(2001)]{Lai-ss01} Lai, T. L. (2001). Sequential Analysis: Some Classical Problems and New Challenges (with Discussion).
\textit{Statistica Sinica} 11: 303–-408.

\bibitem[Lai and Siegmund (1977)]{LaiSieg-77} Lai, T.  L. and Siegmund, D. (1977). A Nonlinear Renewal Theory with Applications to Sequential Analysis~I, \textit{Annals of Statistics} 5: 628--643.

\bibitem[Lai and Siegmund (1979)]{LaiSieg-79} Lai, T.  L. and Siegmund, D. (1979). A Nonlinear Renewal Theory with Applications to Sequential Analysis~II, \textit{Annals of Statistics} 7: 60--76.

\bibitem[Lerche(1986)]{Lerche-AS86} Lerche, H.R. (1986). The Shape of Bayes Tests of Power 1, \textit{Annals of  Statistics} 14: 1030--1048.


\bibitem[Lorden(1967)]{Lorden-AS67} Lorden, G. (1967). Integrated Risk of Asymptotically Bayes Sequential Tests, \textit{Annals of Mathematical  Statistics} 
38: 1399--1422.

\bibitem[Lorden(1971)]{Lorden-AS71} Lorden, G. (1971). Procedures for Reacting to a Change in a Distribution, \textit{Annals of Mathematical Statistics} 42: 1897--1908.

\bibitem[Lorden(1973)]{Lorden-AS73} Lorden, G. (1973). Open-Ended Tests for Koopman--Darmois Families, \textit{Annals of  Statistics} 1: 633--643.

\bibitem[Lorden(1977)]{Lorden-AS77} Lorden, G. (1977). Nearly Optimal Sequential Tests for Finitely Many Parameter Values, \textit{Annals of  Statistics} 5: 1--21.

\bibitem[Lorden and Pollak(2005)]{LordenPollak-AS05} \text{Lorden, G. and Pollak, M.} (2005). Nonanticipating Estimation Applied to Sequential Analysis and Changepoint Detection, \textit{Annals of  Statistics} 3: 1422--1454.


\bibitem[Pavlov(1990)]{Pavlov-TPA90} Pavlov, I. V. (1990). A Sequential Procedure for Testing Composite Hypotheses with Application to the Kiefer--Weiss Problem,  \textit{Theory of Probability and Its Applications} 35: 280–-292.

\bibitem[Pollak(1978)]{Pollak-AS78}  Pollak, M. (1978). Optimality and Almost Optimality of Mixture Stopping Rules, \textit{Annals of Statistics} 6: 910--916.

\bibitem[Pollak and Siegmund(1975)]{PollakSiegmund-AS75} Pollak, M. and Siegmund, D. (1975). Approximations to the Expected  Sample Size of Certain Sequential Tests, \textit{Annals of  Statistics} 3: 1267--1282.

\bibitem[Pollak(1986)]{Pollak-AS86} Pollak, M. (1986). On the Asymptotic Formula for  the Probability of a Type I Error of Mixture-Type Power One Tests, \textit{Annals of  Statistics} 14: 1012--1029.


\bibitem[Pollak and Yakir(1999)]{PollakYakir-SQA99} Pollak, M. and Yakir, B. (1999). A Simple Comparison of Mixture vs.  Nonanticipating Estimation, \textit{Sequential Analysis} 18: 157--164.

\bibitem[Robbins(1970)]{Robbins-AMS70} Robbins, H.(1970). Statistical Methods Related to the Law of the Iterated  Logarithm, 
\textit{Annals of Mathematical Statistics} 41: 1397--1409.

\bibitem[Robbins and Siegmund(1970)]{RobbinsSiegmund-Berkeley70} Robbins, H. and Siegmund, D. (1970).
A Class of Stopping Rules for Testing Parameter Hypotheses, in
{\em Proceedings of the Sixth Berkeley Symposium on
Mathematical Statistics and Probability}, Le Cam, L. M., Neyman, J., and Scott, E. L., editors, June 21--July 18, 1970, vol. 4: Biology and Health, pp. 37--41, Berkeley: University of California Press.

\bibitem[Robbins and Siegmund(1974)]{RobbinsSiegmund-AS74} Robbins, H. and Siegmund, D. (1974). The Expected Sample Size of Some Tests of Power One,  \textit{Annals of Mathematical Statistics}  2: 415--436.

\bibitem[Tartakovsky(1998)]{TartakovskySISP98} Tartakovsky, A. G. (1998). Asymptotic Optimality of Certain Multihypothesis Sequential Tests: Non-i.i.d.\ Case, \textit{Statistical Inference for Stochastic Processes} 1: 265--295.


\bibitem[Tartakovsky et al.(2003)]{Tartakovskyetal-IEEEIT03} Tartakovsky, A. G., Li, X. R., and Yaralov, G. (2003). Sequential Detection of Targets in Multichannel Systems,  \textit{IEEE Transactions on Information Theory}  49: 425--445.

\bibitem[Tartakovsky et al.(2006a)]{Tartakovskyetal-SM06} Tartakovsky, A. G., Rozovskii, B. L., Bla\'zek, R. B., and Kim, H. (2006a). Detection of Intrusions in Information Systems by Sequential Change-Point Methods, \textit{Statistical Methodology} 3: 252--293.

\bibitem[Tartakovsky et al.(2006b)]{Tartakovskyetal-IEEESP06} Tartakovsky, A. G., Rozovskii, B. L., Bla\'zek, R. B., and Kim, H. (2006b). A Novel Approach to Detection of Intrusions in Computer Networks via Adaptive Sequential and Batch-Sequential Change-Point Detection Methods,  \textit{IEEE Transactions on Signal Processing} 54: 3372--3382.

\bibitem[Tartakovsky and Veeravalli(2004)]{TarVeerASM2004} Tartakovsky, A. G. and Veeravalli, V. V. (2004). Change-point Detection in Multichannel and Distributed Systems, in {\em Applied Sequential Methodologies:
Real-World Examples with Data Analysis}, Mukhopadhyay, N., Datta, S., and Chattopadhyay, S., editors, vol. 173 of Statistics: a Series of Textbooks and Monographs, pp. 339--370, New York: Marcel Dekker.

\bibitem[Wald and Wolfowitz(1948)]{WaldWolfowitz48} Wald, A. and Wolfowitz, J. (1948). Optimum Character of the Sequential Probability Ratio Test,  \textit{Annals of Mathematical Statistics}  19: 326--339.

\bibitem[Woodroofe(1982)]{Woodroofe-book82} Woodroofe, M. (1982). \textit{Nonlinear Renewal Theory in Sequential Analysis}, Philadelphia: SIAM.

\end{thebibliography}
\end{document}